\documentclass[leqno,11pt,a4paper]{amsart}

\usepackage[left=2.5cm,right=2.5cm,top=2.5cm,bottom=2.5cm]{geometry}
\usepackage{amsmath}
\usepackage{amsthm}
\usepackage{amssymb}
\usepackage{mathtools}
\usepackage{IEEEtrantools}
\usepackage[shortlabels]{enumitem}
\usepackage{hyperref}
\usepackage{cite}

\usepackage{xcolor}

\theoremstyle{plain}
\newtheorem{theorem}{Theorem}[section]
\newtheorem{prop}[theorem]{Proposition}
\newtheorem{cor}[theorem]{Corollary}
\newtheorem{lem}[theorem]{Lemma}

\theoremstyle{definition}

\newtheorem{remark}[theorem]{Remark}
\newtheorem{definition}[theorem]{Definition}

\numberwithin{equation}{section}

\newcommand{\BC}{{\mathbb{C}}}
\newcommand{\BD}{{\mathbb{D}}}

\newcommand{\BF}{{\mathbb{F}}}

\newcommand{\BP}{{\mathbb{P}}}
\newcommand{\BQ}{{\mathbb{Q}}}
\newcommand{\BR}{{\mathbb{R}}}

\newcommand{\BZ}{{\mathbb{Z}}}

\newcommand{\MA}{\mathcal{A}}

\newcommand{\ME}{\mathcal{E}}

\newcommand{\MJ}{\mathcal{J}}

\newcommand{\MM}{\mathcal{M}}

\newcommand{\MO}{\mathcal{O}}

\newcommand{\MU}{\mathcal{U}}
\newcommand{\MV}{\mathcal{V}}

\newcommand{\on}[1]{\operatorname{#1}}
\newcommand{\wt}[1]{\widetilde{#1}}
\newcommand{\ol}[1]{\overline{#1}}

\begin{document}

\author[O.~Edtmair]{O.~Edtmair}
\address{Department of Mathematics\\University of California at Berkeley\\Berkeley, CA\\94720\\USA}
\email{oliver\_edtmair@berkeley.edu}
\title[An elementary alternative to PFH spectral invariants]{An elementary alternative to PFH spectral invariants}

\begin{abstract}
Inspired by Hutchings' elementary alternative to ECH capacities, we introduce an elementary alternative to spectral invariants defined via periodic Floer homology (PFH). We use these spectral invariants to provide more elementary proofs of a number of results which have recently been obtained using PFH spectral invariants. Among these results are quantitative closing lemmas for area-preserving surface diffeomorphisms and the simplicity conjecture.
\end{abstract}

\maketitle

\tableofcontents

\section{Introduction}

\subsection{Motivation and background}

This paper is motivated by recent work of McDuff-Siegel \cite{MS21} and Hutchings \cite{Hut22}. Both papers provide elementary ersatz definitions of certain sequences of symplectic capacities and show that important applications of the original capacities are recovered by the alternative ones, thus leading to much more elementary proofs. Instead of symplectic capacities, we focus on spectral invariants of area-preserving surface diffeomorphisms. Following the philosophy of \cite{MS21} and \cite{Hut22}, we define an elementary replacement for spectral invariants coming from periodic Floer homology (PFH), whose construction is highly non-trivial and in particular relies on gauge theory. PFH spectral invariants have recently found remarkable applications to $C^0$ symplectic geometry, Hofer geometry and symplectic dynamics in dimension two (see \cite{CHS20}, \cite{CHS21}, \cite{CPZ21}, \cite{EH21}). We provide elementary proofs of some of these applications and significantly reduce the amount of sophisticated machinery used in the background. In particular, we do not appeal to Floer theory at all. If the underlying surface has genus zero, we completely remove Seiberg-Witten theory from our proofs.\\

Let us begin with a brief review of Hutchings' elementary capacities defined in \cite{Hut22}. Recall that a {\it symplectic capacity} $c$ is a function that assigns numbers $c(X,\omega)\in [0,\infty]$ to symplectic manifolds $(X,\omega)$. Symplectic capacities are required to be monotonic under symplectic embeddings and behave linearly with respect to scalings of the symplectic form. More precisely, one requires:
\begin{itemize}
\item (Monotonicity) If $(X,\omega)$ symplectically embeds into $(X',\omega')$, then $c(X,\omega)\leq c(X',\omega')$.
\item (Conformality) For every $r>0$, we have $c(X,r\omega)=rc(X,\omega)$.
\end{itemize}
We refer to Cieliebak-Hofer-Latschev-Schlenk \cite{CHLS07} for a survey on symplectic capacities. Hutchings' elementary capacities are defined on $4$-dimensional symplectic manifolds. They form a family of capacities $c_k^H$ parametrized by positive integers $k$. The construction of $c_k^H$ is based on pseudo-holomorphic curves. Roughly speaking, the value $c_k^H(X,\omega)$ is defined to be the infimum over all real numbers numbers $C$ such that for every compatible almost complex structure $J$ on $(X,\omega)$ and every choice $P\subset X$ of $k$ points, there exists a $J$-holomorphic curve $u$ of symplectic area $\ME(u)$ at most $C$ passing through $P$. More precisely, if $(X,\omega)$ is a closed symplectic $4$-manifold, then
\begin{equation}
\label{eq:def_hutchings_elementary}
c_k^H(X,\omega) \coloneqq \sup_{J,P} \enspace\inf_u \enspace \ME(u)
\end{equation}
where the supremum is taken over all compatible almost complex structures $J$ and all collections of $k$ points $P\subset X$, and the infimum is taken over all compact $J$-curves $u$ passing through $P$ which are non-constant on all components. The symplectic area $\ME(u)$ simply is the integral of $\omega$ over $u$. The definition of $c_k^H$ extends to the case of compact $4$-dimensional Liouville domains $(X,\lambda)$ as follows. We form the completion $\ol{X}$ of $X$ by attaching a cylindrical end $[0,\infty)\times\partial X$ along the contact boundary $(\partial X,\lambda)$ of $X$. The supremum in \eqref{eq:def_hutchings_elementary} is now taken over all sets $P\subset X$ of $k$ points and all compatible almost complex structures on $\ol{X}$ which are in addition symplectization admissible on the cylindrical end $[0,\infty)\times\partial X$. The infimum runs over all $J$-holomorphic curves $u$ in $\ol{X}$ passing through $P$. The domains of these curves are compact Riemann surfaces with finitely many punctures removed and the curves are required to be asymptotic to cylinders $[0,\infty)\times \gamma$ over periodic Reeb orbits $\gamma$ of $(\partial X,\lambda)$ at these punctures. The symplectic area $\ME(u)$ in this setting is replaced by the sum over all the actions $\MA(\gamma)=\int_\gamma\lambda$ of asymptotic Reeb orbits $\gamma$ of $u$. Hutchings shows that the functions $c_k^H$ indeed are symplectic capacities. Moreover, he proves that they satisfy a list of properties analogous to the properties of ECH capacties, which allows for more elementary proofs of important applications of these capacities.

\subsection{Elementary spectral invariants}
\label{subsec:elementary_spectral_invariants}

Let us now sketch how the construction of $c_k^H$ can be adapted to define spectral invariants of area-preserving diffeomorphisms. Before explaining the definition, we introduce our geometric setup. Unless stated otherwise, $\Sigma$ will denote a closed, connected, orientable surface of genus $g$ throughout this paper. We equip it with an area form $\omega$ of total area $A$. Moreover, we fix an area-preserving diffeomorphism $\phi$ of $(\Sigma,\omega)$. The {\it mapping torus} $Y_\phi$ of $\phi$ is the closed $3$-manifold defined by
\begin{equation}
Y_\phi\coloneqq \frac{[0,1]\times\Sigma}{\sim}\qquad (1,x)\sim (0,\phi(x)) \enspace \text{for}\enspace x\in\Sigma.
\end{equation}
It naturally fibres over the circle $\BR/\BZ$ with fibre $\Sigma$. Let $t$ denote the coordinate on the interval $[0,1]$. The vector field $\partial_t$ on $[0,1]\times\Sigma$ descends to a smooth vector field on $Y_\phi$ transverse to the fibres. Let $\omega_\phi$ denote the unique closed $2$-form on $Y_\phi$ which vanishes on $\partial_t$ and whose restriction to the fibres of $Y_\phi\rightarrow \BR/\BZ$ agrees with $\omega$. The dynamics of the vector field $\partial_t$ is closely related to the dynamics of the surface diffeomorphism $\phi$. Indeed, the fibre $\{0\}\times\Sigma$ of $Y_\phi$ is a global surface of section and the first return map is given by $\phi$. The {\it symplectization} of $Y_\phi$ is the symplectic $4$-manifold $(M_\phi,\Omega_\phi)$ defined by
\begin{equation}
\label{eq:symplectization_definition}
M_\phi\coloneqq \BR\times Y_\phi\qquad \text{and}\qquad \Omega_\phi \coloneqq ds\wedge dt + \omega_\phi
\end{equation}
where $s$ denotes the coordinate on $\BR$.\\

We are interested in Hamiltonian deformations of $\phi$. Given a time-dependent Hamiltonian $H:[0,1]\times\Sigma\rightarrow\BR$, the induced Hamiltonian vector field $X_H$ is characterized by
\begin{equation}
\label{eq:hamiltonian_vector_field}
\iota_{X_{H_t}}\omega = dH_t.
\end{equation}
We let $\phi_H^t$ denote the Hamiltonian flow generated by $X_H$ and abbreviate
\begin{equation}
\label{eq:hamiltonian_deformation}
\phi_H\coloneqq \phi\circ\phi_H^1.
\end{equation}
It will be convenient to view functions $H\in C^\infty(Y_\phi)$ as time dependent Hamiltonians on $\Sigma$ by pulling them back to $[0,1]\times\Sigma$ via the natural projection $[0,1]\times\Sigma\rightarrow Y_\phi$. Given $H\in C^\infty(Y_\phi)$, we may identify the mapping torus $Y_{\phi_H}$ with the graph $\on{gr}(H)$, which we regard as a hypersurface inside the symplectization $M_\phi$. Indeed, let us define the diffeomorphism
\begin{equation}
\label{eq:diffeomorphism_between_symplectizations}
f_H:M_{\phi_H}\rightarrow M_\phi\quad (s,[t,x])\mapsto (s+H(t,\phi_H^t(x)),[t,\phi_H^t(x)]).
\end{equation}
It maps the mapping torus $Y_{\phi_H} \cong \{0\}\times Y_{\phi_H}$ to $\on{gr}(H)$. A direct computation shows that $f_H^*\Omega_\phi = \Omega_{\phi_H}$, i.e.\ $f_H$ is a symplectomorphism between the symplectizations of $Y_{\phi_H}$ and $Y_\phi$. In particular, this implies that $f_H$ maps the characteristic foliation on $Y_{\phi_H}$, which is tangent to $\on{ker}\omega_{\phi_H}$, to the characteristic foliation on $\on{gr}(H)$, which is tangent to $\on{ker}\Omega_\phi|_{\on{gr}(H)}$.\\

Given a Hamiltonian $H\in C^\infty(Y_\phi)$, let us introduce the notation $\on{gr}^\pm(H)$ for the epi- and subgraph of $H$, respectively. In other words, $\on{gr}^\pm(H)$ are the closures of the regions in $M_\phi$ above and below the graph of $H$, respectively. If $H_\pm\in C^\infty(Y_\phi)$ are Hamiltonians satisfying $H_+>H_-$, we let $M_{H_+,H_-}$ denote the closure of the bounded region in $M_\phi$ between the graphs $\on{gr}(H_+)$ and $\on{gr}(H_-)$, i.e.
\begin{equation}
M_{H_+,H_-}\coloneqq \on{gr}^+(H_-)\cap \on{gr}^-(H_+).
\end{equation}
Via the identification between mapping tori and graphs of Hamiltonians, $M_{H_+,H_-}$ can be viewed as a symplectic cobordism between the mapping tori $Y_{\phi_{H_\pm}}$. The symplectomorphism $f_{H_+}$ exhibits $\on{gr}^+(H_+)$ as a positive cylindrical end $[0,\infty)\times Y_{\phi_{H_+}}$ attached to the symplectic cobordism $M_{H_+,H_-}$. Similarly, we regard $\on{gr}^-(H_-)$ as a negative cylindrical end $(-\infty,0]\times Y_{\phi_{H_-}}$ via $f_{H_-}$. This identifies $M_\phi$ with the open symplectic manifold obtained from $M_{H_+,H_-}$ by attaching cylindrical ends along the boundary components $Y_{\phi_{H_\pm}}$, which we will refer to as the {\it symplectic completion} of the cobordism $M_{H_+,H_-}$.\\

Let us fix two Hamiltonians $H_\pm\in C^\infty(Y_\phi)$ and assume that $H_+>H_-$. For every positive integer $d\geq 1$ and every non-negative integer $k\geq 0$, we define a spectral invariant $c_{d,k}^\phi(H_+,H_-)$ based on pseudo-holomorphic curves in the completion of the cobordism $M_{H_+,H_-}$. Our definition is modeled on the definition of $c_k^H$ given in equation \eqref{eq:def_hutchings_elementary}:
\begin{equation}
c_{d,k}^\phi(H_+,H_-) \coloneqq \sup_{J,P} \enspace\inf_u \enspace \ME_{H_+,H_-}(u)
\end{equation}
The supremum is taken over all compatible almost complex structures $J$ on the symplectic completion of $M_{H_+,H_-}$. We require $J$ to be symplectization admissible (see Definition \ref{def:symplectization_admissible_almost_complex_structure}) on the cylindrical ends of the completion outside some compact set (which is allowed to depend on $J$). Moreover, we take the supremum over all sets $P$ of $k$ points in the completion. The points are not restricted to $M_{H_+,H_-}$ and may be contained in the cylindrical ends. The infimum runs over all $J$-holomorphic curves $u$ in the completion which pass through $P$ and all of whose components are non-constant. The domain of $u$ is required to be a compact Riemann surface with finitely many punctures removed. At these punctures, $u$ is assumed to be asymptotic to cylinders of the form $[0,\infty)\times\gamma_+$ or $(-\infty,0]\times\gamma_-$ where $\gamma_\pm$ are periodic orbits of the vector field $\partial_t$ on the mapping tori $Y_{\phi_{H_\pm}}$, respectively. We demand that the {\it degree} of $u$, i.e.\ the intersection number of $u$ with a fiber of the projection $M_\phi\rightarrow\BR\times\BR/\BZ$ is equal to $d$. The symplectic area $\ME_{H_+,H_-}(u)$ is defined by
\begin{equation}
\ME_{H_+,H_-}(u) \coloneqq \int_{u\cap [0,\infty)\times Y_{\phi_{H_+}}} \omega_{\phi_{H_+}} + \int_{u\cap M_{H_+,H_-}} \Omega_\phi + \int_{u\cap (-\infty,0]\times Y_{\phi_{H_-}}} \omega_{\phi_{H_-}}.
\end{equation}
The symplectic cobordism $M_{H_+,H_-}$ is only defined if $H_+>H_-$. However, there is natural way of extending the spectral invariants $c_{d,k}^\phi$ to arbitrary pairs of Hamiltonians (see section \ref{sec:elemantary_construction_of_spectral_invariants}). Thus our spectral invariants take the form of functions
\begin{equation}
c_{d,k}^\phi : C^\infty(Y_\phi)\times C^\infty(Y_\phi) \longrightarrow \BR\cup\{+\infty\}.
\end{equation}
If there is no danger of confusion, we will drop $\phi$ from the notation and simply write $c_{d,k}$. We will see that for fixed $d$ and $k$, the invariant $c_{d,k}$ is either equal to $+\infty$ for all pairs of Hamiltonians or takes values in $\BR$.\\

Before stating the main properties of $c_{d,k}$, we need some more preparation. An {\it orbit set} of the mapping torus $Y_\phi$ is a finite set $\alpha = \{(\alpha_i,m_i)\}$ consisting of tuples $(\alpha_i,m_i)$ where the $\alpha_i$ are distinct simple periodic orbits of the vector field $\partial_t$  on $Y_\phi$ and the $m_i$ are positive integers. We may regard such an orbit set $\alpha$ as a $1$-chain inside $Y_\phi$ and define its {\it degree} to be the intersection number with a fibre of $Y_\phi\rightarrow\BR/\BZ$. Consider arbitrary Hamiltonians $H_\pm\in C^\infty(Y_\phi)$. Let $\alpha_\pm$ be orbit sets of $Y_{\phi_{H_\pm}}$, respectively. Via the identification of $Y_{\phi_{H_\pm}}$ with the graphs $\on{gr}(H_\pm)$, we can regard $\alpha_\pm$ as $1$-cycles in $M_\phi$. If they are homologous, we let $H_2(M_\phi,\alpha_+,\alpha_-)$ denote the set of all $2$-chains $Z$ in $M_\phi$ with integer coefficients satisfying $\partial Z=\alpha_+-\alpha_-$. Two such chains are considered equivalent if their difference is null-homologous in $M_\phi$. The set $H_2(M_\phi,\alpha_+,\alpha_-)$ is naturally an affine space over $H_2(M_\phi;\BZ)\cong H_2(Y_\phi;\BZ)$.

\begin{definition}
\label{def:relative_action_spectrum}
Given a positive integer $d$ and Hamiltonians $H_\pm\in C^\infty(Y_\phi)$, we define the {\it relative action spectrum in degree $d$}, denoted by $\on{Spec}_d(H_+,H_-)$, to be the set of all numbers $\langle Z,\Omega_\phi \rangle$ where $Z \in H_2(M_\phi,\alpha_+,\alpha_-)$ and $\alpha_\pm$ are degree $d$ orbit sets of $Y_{\phi_{H_\pm}}$, respectively.
\end{definition}

The relative action spectrum is always a subset of $\BR$ of measure zero (see Lemma \ref{lem:relative_action_spectrum_measure_zero} in appendix \ref{sec:properties_of_the_relative_action_spectrum}). It need not be closed in general. For example, it can happen that the image of the homomorphism $\langle \cdot, [\omega_\phi] \rangle : H_2(Y_\phi;\BZ)\rightarrow\BR$ is dense. In this case, the relative action spectrum is also dense whenever it is non-empty. Following \cite{EH21}, we make the following definition.

\begin{definition}
\label{def:rational_diffeomorphism}
An area-preserving diffeomorphism $\phi$ of $(\Sigma,\omega)$ is called {\it rational} if the cohomology class $[\omega_\phi]\in H^2(Y_\phi;\BR)$ is a real multiple of a cohomology class in $H^2(Y_\phi;\BQ)$.
\end{definition}

We remark that the property of being rational only depends on the Hamiltonian isotopy class of $\phi$ and that the set of rational area-preserving diffeomorphism is $C^\infty$ dense inside the space of all area-preserving diffeomorphisms. If $\phi$ is rational, then the relative action spectrum must be closed (see Lemma \ref{lem:relative_action_spectrum_closed} in appendix \ref{sec:properties_of_the_relative_action_spectrum}).

\begin{theorem}
\label{theorem:spectral_invariants_properties_basic}
The spectral invariants $c_{d,k}$ satisfy the following properties. Let $H_\pm\in C^\infty(Y_\phi)$ denote Hamiltonians and fix integers $d\geq 1$ and $k\geq 0$.
\begin{enumerate}
\item {\bf (Shift)} Let $C_\pm\in\BR$. Then 
\begin{equation*}
c_{d,k}(H_++C_+,H_-+C_-) = c_{d,k}(H_+,H_-) + d(C_+-C_-).
\end{equation*}
\item {\bf (Increasing)} If $k'\geq k$, then $c_{d,k'}(H_+,H_-)\geq c_{d,k}(H_+,H_-)$.
\item {\bf (Sublinearity)} Suppose that $d=d_1+d_2$ and $k=k_1+k_2$. Then
\begin{equation*}
c_{d,k}(H_+,H_-) \leq c_{d_1,k_1}(H_+,H_-) + c_{d_2,k_2}(H_+,H_-).
\end{equation*}
\item {\bf (Non-negativity)} If $H_+\geq H_-$, then $c_{d,k}(H_+,H_-)\geq 0$.
\item {\bf (Spectrality)} If $\phi$ is rational and $c_{d,k}$ is finite, then $c_{d,k}(H_+,H_-) \in \on{Spec}_d(H_+,H_-)$.
\item {\bf (Continuity)} Endow $C^\infty(Y_\phi)$ with the $C^0$ norm. Then $c_{d,k}$ is Lipschitz continuous with Lipschitz constant $d$ in both entries.
\item {\bf (Monotonicity)} Let $H_0\in C^\infty(Y_\phi)$ be any Hamiltonian. Suppose that $k=k_++k_-$ for non-negative integers $k_\pm$. Then
\begin{equation*}
c_{d,k}(H_+,H_-) \geq c_{d,k_+}(H_+,H_0) + c_{d,k_-}(H_0,H_-).
\end{equation*}
\item {\bf (Packing)} Suppose that $H_+ > H_-$ and let $X$ denote a compact Liouville domain which admits a symplectic embedding into the symplectic cobordism $M_{H_+,H_-}$ between the graphs $\on{gr}(H_\pm)$. Then
\begin{equation*}
c_{d,k}(H_+,H_-)\geq c_k^H(X).
\end{equation*}
\item {\bf (Closed curve upper bound)} Let $(X,\Omega)$ be a closed symplectic $4$-manifold. Assume that $H_+>H_-$ and that there exists a symplectic embedding $\iota:M_{H_+,H_-} \rightarrow X$. Let $Z\in H_2(X;\BZ)$ be a homology class and assume that the Gromov-Taubes invariant $\on{Gr}(X,\Omega;Z)$ does not vanish (see appendix \ref{subsec:review_taubes_gromov_invariants} for a review of $\on{Gr}$). Suppose that the ECH index $I(Z)$ is equal to $2k$ and that the intersection number $Z\cdot \iota_*[\Sigma]$ is equal to $d$. Then
\begin{equation*}
c_{d,k}(H_+,H_-)\leq \langle Z,[\Omega] \rangle.
\end{equation*}
\end{enumerate}
\end{theorem}

\begin{remark}
If $c_{d,k}$ is constant equal to $+\infty$, then the properties in Theorem \ref{theorem:spectral_invariants_properties_basic} (except for the closed curve upper bound) are trivially true. We point out that at this stage it is not clear that there exist area-preserving diffeomorphisms $\phi$ such that $c_{d,k}$ takes values in $\BR$. We will see in subsection \ref{subsec:relationship_with_PFH} below that for all rational $\phi$ there exist arbitrarily large $d$ such that the invariant $c_{d,k}$ is indeed finite for all $k$.
\end{remark}

\begin{remark}
The main content of Theorem \ref{theorem:spectral_invariants_properties_basic} is monotonicity, the packing property and the closed curve upper bound. The remaining properties (except for continuity) follow in a straightforward manner from our construction of $c_{d,k}$. Continuity is an easy consequence of the shift property, non-negativity and monotonicity. The proofs of monotonicity, the packing property and the closed curve upper bound are based on a neck stretching argument similar to the one explained in \cite{Hut22}. The main ingredients are a local version of Gromov compactness due to Taubes (see \cite[Proposition 3.3]{Tau98}) and a relative adjunction formula and asymptotic writhe bound for pseudo-holomorphic curves in symplectic $4$-manifolds with cylindrical ends due to Hutchings (see \cite[Proposition 4.9 and Lemma 4.20]{Hut09}).
\end{remark}

\subsection{Weyl law}

Inspired by the Weyl law for PFH spectral invariants proved independently by Cristofaro-Gardiner-Prasad-Zhang in \cite{CPZ21} and Edtmair-Hutchings in \cite{EH21}, we define a second family of spectral invariants. For every positive integer $d\geq g$, we define
\begin{equation*}
c_d^\phi:C^\infty(Y_\phi)\times C^\infty(Y_\phi)\longrightarrow\BR\cup\{+\infty\}
\end{equation*}
by
\begin{equation}
\label{eq:definition_spectral_invariant_weyl}
c_d^\phi(H_+,H_-) \coloneqq \sup_{k\in\BZ_{\geq 0}} \left( c_{d,k(d-g+1)}^\phi(H_+,H_-) - k A \right).
\end{equation}
Recall that $g$ and $A$ denote the genus and area of $(\Sigma,\omega)$, respectively. We refer to section \ref{subsec:relationship_with_PFH} for some motivation for this definition. We will abbreviate $c_d=c_d^\phi$ if this does not cause ambiguity. Similarly to the invariants $c_{d,k}$, for every fixed $d$, the invariant $c_d$ is either equal to $+\infty$ on all pairs of Hamiltonians or takes values in $\BR$. Moreover, finiteness of $c_d$ implies that $c_{d,k}$ is finite for all $k\geq 0$.

\begin{theorem}
\label{theorem:spectral_invariants_properties_weyl}
The spectral invariants $c_d$ satisfy the following properties. Let $H_\pm\in C^\infty(Y_\phi)$ denote Hamiltonians and fix a positive integer $d\geq g$.
\begin{enumerate}
\item {\bf (Shift)} Let $C_\pm\in\BR$. Then 
\begin{equation*}
c_d(H_++C_+,H_-+C_-) = c_d(H_+,H_-) + d(C_+-C_-).
\end{equation*}
\item {\bf (Non-negativity)} If $H_+\geq H_-$, then $c_d(H_+,H_-)\geq 0$.
\item {\bf (Vanishing)} If $c_d$ is finite, then $c_d(H,H)=0$ for all $H\in C^\infty(Y_\phi)$.
\item {\bf (Monotonicity)} Let $H_0\in C^\infty(Y_\phi)$ be any Hamiltonian. Then
\begin{equation*}
c_d(H_+,H_-) \geq c_d(H_+,H_0) + c_d(H_0,H_-).
\end{equation*}
\item {\bf (Continuity)} Endow $C^\infty(Y_\phi)$ with the $C^0$ norm. Then $c_d$ is Lipschitz continuous with Lipschitz constant $d$ in both entries.
\item {\bf (Spectrality)} If $\phi$ is rational and is $c_d$ finite, then $c_d(H_+,H_-) \in \on{Spec}_d(H_+,H_-)$.
\item {\bf (Weyl law)} Let $(d_i)_{i>0}$ be a sequence of positive integers such that $\lim_{i\rightarrow\infty} d_i = +\infty$. Suppose that for every $i$ the spectral invariant $c_{d_i}$ is finite. Then
\begin{equation*}
\lim_{i\rightarrow\infty} \frac{c_{d_i}(H_+,H_-)}{d_i} = A^{-1} \int_{Y_\phi} (H_+-H_-) dt\wedge\omega_\phi.
\end{equation*}
\end{enumerate}
\end{theorem}

\begin{remark}
\label{remark:weyl_law_elementary}
The shift property, non-negativity, vanishing, monotonicity, continuity and spectrality are easy consequences of the corresponding properties of the spectral invariants $c_{d,k}$. The Weyl law requires one additional ingredient, namely the asymptotics of Hutchings' elementary capacities $c_k^H$ for the ball. The proof of the asymptotic formula for $c_k^H$ given in \cite{Hut22} is elementary in the sense that it does not require ingredients beyond Gromov's paper \cite{Gr85}, Taubes local version of Gromov compactness \cite{Tau98}, the relative relative adjunction formula and the writhe bound \cite{Hut09}. We emphasize that as soon as finiteness of the spectral invariants $c_d$ is established, the Weyl law for $c_d$ is elementary in the same sense.
\end{remark}

\subsection{Finiteness}
\label{subsec:finiteness}

In order for the spectral invariants $c_{d,k}$ and $c_d$ to be useful, we need results that guarantee their finiteness. It is possible to bound $c_{d,k}$ and $c_d$ from above using PFH spectral invariants (see section \ref{subsec:relationship_with_PFH} below). However, in some sense this defeats the whole purpose of this paper, which is reproving applications of PFH spectral invariants by more elementary methods. Thus we need alternative ways to show finiteness of $c_{d,k}$ and $c_d$. Our main strategy is to use the closed curve upper bound in Theorem \ref{theorem:spectral_invariants_properties_basic}. This requires proving the existence of pseudo-holomorphic curves in closed symplectic $4$-manifolds. Our main source of pseudo-holomorphic curves is the following non-vanishing theorem for Gromov-Taubes invariants (see appendix \ref{subsec:review_taubes_gromov_invariants} for a brief review).

\begin{theorem}
\label{theorem:non_vanishing_taubes_gromov}
Let $\sigma$ be an area form on $S^2$. We endow the product $X\coloneqq \Sigma\times S^2$ with the product symplectic form $\Omega\coloneqq \omega\oplus\sigma$. Let $(c,d)\in\BZ^2$ and define the homology class
\begin{equation}
\label{eq:non_vanishing_taubes_gromov_definition_homology_class}
Z\coloneqq c \cdot [\Sigma\times *] + d \cdot [*\times S^2] \in H_2(X;\BZ).
\end{equation}
Then the Gromov-Taubes invariant $\on{Gr}(X,\Omega;Z)$ is non-vanishing if and only if $c\geq 0$ and $d\geq (g-1)\frac{c}{c+1}$.
\end{theorem}

\begin{remark}
We provide a proof of Theorem \ref{theorem:non_vanishing_taubes_gromov} in appendix \ref{sec:computation_of_taubes_gromov_invariants}. In the case $\Sigma=S^2$, our proof is elementary and in particular does not use Seiberg-Witten theory. In fact, the existence of the relevant pseudo-holomorphic curves in $S^2\times S^2$ was essentially already proved by Gromov \cite[Section 0.2.A.]{Gr85}, although he does not explicitly mention point constraints. The upshot is that all results we prove in the genus zero case are particularly minimalistic in terms of ingredients going into their proofs. For general $\Sigma$, our argument is based on Taubes' ``Seiberg-Witten = Gromov" theorem (see \cite{Tau00}). One might hope to be able to avoid this machinery.
\end{remark}

The following result is a straightforward consequence of Theorem \ref{theorem:non_vanishing_taubes_gromov} and the closed curve upper bound in Theorem \ref{theorem:spectral_invariants_properties_basic}. See Theorem \ref{theorem:finiteness_for_hamiltonian_diffeomorphisms} in subsection \ref{subsec:proofs_of_basic_properties} for a proof.

\begin{theorem}
Suppose that $\phi$ is a Hamiltonian diffeomorphism of $(\Sigma,\omega)$. Then the invariants $c_{d,k}^\phi$ and $c_d^\phi$ are finite for all positive integers $d\geq g$ and non-negative integers $k$.
\end{theorem}

\begin{remark}
Using PFH one can show that a similar finiteness result continues to hold for all rational $\phi$ (see Corollary \ref{cor:finiteness_rational_phi}). It would be interesting to understand to what extent this can also be proved by more elementary methods.
\end{remark}

\subsection{Relationship with PFH spectral numbers}
\label{subsec:relationship_with_PFH}

Similarly to the upper bound for Hutchings' alternative capacities in terms of ECH capacities (see \cite{Hut22}), it is possible to bound $c_{d,k}$ and $c_d$ from above using PFH spectral invariants. We begin with a brief overview of PFH and refer to \cite{EH21} for more details. Given an area-preserving diffeomorphism $\phi$ of $(\Sigma,\omega)$, one can define a module $HP(\phi,\gamma,G)$, called the {\it periodic Floer homology} of $\phi$. In addition to $\phi$, it depends on the choice of a $1$-cycle $\gamma\subset Y_\phi$, called {\it reference cycle}. Moreover, it depends on the choice of a subgroup $G\subset\operatorname{ker}([\omega_\phi])$ of the kernel of the homomorphism $\langle  \cdot,[ \omega_\phi ]\rangle: H_2(Y_\phi;\BZ)\rightarrow\BR$. Here we simply take $G=\ker([\omega_\phi])$ and omit it from the notation. Periodic Floer homology is a module over a Novikov ring $\Lambda$ which is a suitable completion of the group ring $\BF_2[H_2(Y_\phi,\BZ)/\on{ker}([\omega_\phi])]$. It is equipped with an endomorphism $U$ called the $U$-map. Moreover, periodic Floer homology naturally carries the more refined structure of a persistence module $(HP^L(\phi,\gamma))_{L\in\BR}$ arising from a chain level action filtration. Using this structure, one can define a {\it spectral invariant} $c_\sigma(\phi,\gamma)\in\BR$ associated to any non-zero class $\sigma\in HP(\phi,\gamma)$. In order to explain the relationship of these spectral invariants with periodic orbits of $\phi$, let us make the following definitions. An {\it anchored orbit set} is a tuple $(\alpha,Z)$ consisting of an orbit set $\alpha$ of $Y_\phi$ and a relative homology class $Z\in H_2(Y_\phi,\alpha,\gamma)$. The {\it action} of an anchored orbit set $(\alpha,Z)$ is defined to be
\begin{equation}
\MA(\alpha,Z) \coloneqq \int_Z \omega_\phi.
\end{equation}
If $\phi$ is rational, the spectral invariant $c_\sigma(\phi,\gamma)$ has the following important property: There exists an anchored orbit set $(\alpha,Z)$ such that $\mathcal{A}(\alpha,Z)= c_\sigma(\phi,\gamma)$.\\
Let us now turn to the behaviour of $HP(\phi,\gamma)$ and the associated spectral invariants $c_\sigma(\phi,\gamma)$ under Hamiltonian deformations of $\phi$. It is proved in \cite{EH21} that for any Hamiltonian $H$, the periodic Floer homology $HP(\phi_H,\gamma_H)$ is naturally isomorphic to $HP(\phi,\gamma)$. Here $\gamma_H$ is a reference cycle in $Y_{\phi_H}$ determined by $\gamma$ and $H$. Roughly speaking, these natural isomorphisms are constructed by identifying mapping tori $Y_{\phi_H}$ with graphs $\on{gr}(H)$ in the symplectization $M_\phi$ as explained in subsection \ref{subsec:elementary_spectral_invariants} and considering cobordism maps. Given a class $\sigma\in HP(\phi,\gamma)$, let $\sigma_H$ denote the corresponding class in $HP(\phi_H,\gamma_H)$. Following \cite{EH21}, we introduce the notation
\begin{equation}
c_\sigma(\phi,\gamma,H)\coloneqq c_{\sigma_H}(\phi_H,\gamma_H).
\end{equation}
These spectral invariants are related to our elementary spectral invariants $c_{d,k}$ as follows: Let $\gamma\subset Y_\phi$ be a reference cycle of positive degree $d>0$. Let $\sigma\in HP(\phi,\gamma)\setminus\{0\}$ be a non-vanishing class and assume that $U^k\sigma \neq 0$ for some non-negative integer $k$. Then for all Hamiltonians $H_\pm\in C^\infty(Y_\phi)$ we have
\begin{equation}
\label{eq:relationship_spectral_numbers_elemantary_PFH}
c_{d,k}(H_+,H_-)\leq c_\sigma(\phi,\gamma,H_+) -c_{U^k\sigma}(\phi,\gamma,H_-)+ \int_\gamma(H_+-H_-)dt.
\end{equation}
This follows from an argument similar to the one given in \cite{Hut22} to prove that the capacities $c_k^H$ are bounded from above by the ECH capacities. The rough idea is that the $U$-map on PFH is defined by counting pseudo-holomorphic curves in the symplectization $M_\phi$ with point constraints. Thus a non-trivial $U$-map between PFH classes of a certain action difference implies the existence of pseudo-holomorphic curves with point constraints and with symplectic area bounded by the action difference.

\begin{remark}
From this perspective, PFH should be regarded as a sophisticated tool for producing holomorphic curves. One of the main points of this paper is to show that in many interesting situations, holomorphic curves relevant for applications can be obtained by more elementary methods.
\end{remark}

In particular, it follows from \eqref{eq:relationship_spectral_numbers_elemantary_PFH} that the existence of a degree $d$ reference cycle $\gamma$ with the property that there exists a PFH class $\sigma$ such that $U^k\sigma\neq 0$ implies that $c_{d,k}$ is finite. Fortunately, such classes $\sigma$ are abundant, at least if $\phi$ is rational. We recall the following definition from \cite{EH21}.

\begin{definition}
\label{def:u_cyclic}
A nonzero class $\sigma\in HP(\phi,\gamma)$ is called $U$-cyclic if there exists a positive integer $m$ such that
\begin{equation}
U^{m(d-g+1)}\sigma = q^{-m[\Sigma]}\sigma.
\end{equation}
Here $d$ denotes the degree of $\gamma$ and $g$ the genus of $\Sigma$. Moreover, $q$ is the formal variable of the Novikov ring $\Lambda$.
\end{definition}

Clearly, the existence of a $U$-cyclic class in degree $d$ implies that $c_{d,k}$ is finite for all non-negative integers $k$. We have the identity
\begin{equation}
\label{eq:novikov_twisting_spectral_invariant}
c_{q^{-[\Sigma]}\sigma}(\phi,\gamma,H) = c_\sigma(\phi,\gamma,H) - A
\end{equation}
where $A$ is the area of $(\Sigma,\omega)$. It follows from \eqref{eq:novikov_twisting_spectral_invariant} and \eqref{eq:relationship_spectral_numbers_elemantary_PFH} that for every $U$-cyclic class $\sigma$ we have
\begin{equation}
\label{eq:upper_bound_elementary_pfh_for_u_cyclic}
c_{d,k(d-g+1)}(H_+,H_-) \leq c_\sigma(\phi,\gamma,H_+) -c_\sigma(\phi,\gamma,H_-)+ \int_\gamma(H_+-H_-)dt + kA
\end{equation}
whenever $k$ is an integer multiple of the integer $m$ from Definition \ref{def:u_cyclic}. It follows easily from \eqref{eq:upper_bound_elementary_pfh_for_u_cyclic} and the definition of $c_d$ in \eqref{eq:definition_spectral_invariant_weyl} that the existence of a $U$-cyclic element in degree $d$ implies that $c_d$ is finite. In fact, \eqref{eq:upper_bound_elementary_pfh_for_u_cyclic} can be regarded as a motivation for the definition of $c_d$.\\

The existence of an abundance of $U$-cyclic PFH classes was proved by Cristofaro-Gardiner-Pomerleano-Prasad-Zhang in \cite{CPPZ21} using Lee-Taubes' isomorphism between PFH and monopole Floer homology \cite{LT12} and a computation of monopole Floer homology based on Kronheimer-Mrowka's book \cite{KM08}. In order to state their result, we make the following definition.

\begin{definition}
\label{definition:monotone_reference_cycle}
Let $E\subset TY_\phi$ denote the vertical tangent bundle of the fibre bundle $Y_\phi\rightarrow\BR/\BZ$. A reference cycle $\gamma\subset Y_\phi$ is called {\it monotone} if $[\omega_\phi]\in H^2(Y_\phi;\BR)$ is a real multiple of the image of the cohomology class $c_1(E)+2 \on{PD}([\gamma])\in H^2(Y_\phi;\BZ)$.
\end{definition}

\begin{remark}
The existence of a monotone reference cycle $\gamma$ implies that $\phi$ is rational. Conversely, for every rational $\phi$ there exist monotone reference cycles of arbitrarily large degree. In fact, the degrees of monotone reference cycles form an infinite arithmetic progression.
\end{remark}

\begin{theorem}[Existence of $U$-cyclic elements \cite{CPPZ21}]
Suppose that $\gamma\subset Y_\phi$ is a monotone reference cycle of degree greater than $\max\{2g-2,0\}$. Then $HP(\phi,\gamma)$ contains $U$-cyclic elements.
\end{theorem}

\begin{cor}
\label{cor:finiteness_rational_phi}
Suppose that $\phi$ is rational. Then there exist arbitrarily large positive integers $d$ such that $c_{d,k}^\phi$ and $c_d^\phi$ are finite for all $k$.
\end{cor}

\subsection{Quantitative closing lemmas}

In this section, we explain an application of our elementary spectral invariants to the dynamics of area-preserving surface diffeomorphisms. More specifically, we are able to reprove certain cases of a quantitative closing lemma which was recently established by Edtmair-Hutchings in \cite{EH21} (see also Cristofaro-Gardiner-Prasad-Zhang \cite{CPZ21} for an independent proof of closing lemmas). Before stating the quantitative closing lemma, we need the following definition from \cite{EH21}.

\begin{definition}
\label{def:ual_admissible_hamiltonians}
Let $\MU\subset\Sigma$ be a nonempty open set, let $l\in (0,1)$, and let $a\in (0,\on{area}(\MU))$. A $(\MU,a,l)$-{\it admissible Hamiltonian} is a smooth function $H:[0,1]\times \Sigma\rightarrow \BR$ such that:
\begin{itemize}
\item $H(t,x) = 0$ for $t$ close to $0$ or $1$.
\item $H(t,x) = 0$ for $x\notin\MU$.
\item $H\geq 0$
\item There is an interval $I\subset (0,1)$ of length $l$ and a disk $D\subset \MU$ of area $a$ such that $H\geq 1$ on $I\times D$.
\end{itemize}
\end{definition}

The following result is a special case of the quantitative closing lemma  \cite[Theorem 7.4]{EH21}. We give an independent elementary proof.

\begin{theorem}
\label{theorem:quantitative_closing_lemma_hamiltonian_diffeomorphisms}
Let $\phi$ be a Hamiltonian diffeomorphism of $(\Sigma,\omega)$. Let $\MU\subset\Sigma$ be a non-empty open subset and let $H$ be a $(\MU,a,l)$-admissible Hamiltonian. If $0<\delta\leq al^{-1}$, then for some $\tau\in [0,\delta]$, the map $\phi_{\tau H}$ has a periodic orbit intersecting $\MU$ with period d satisfying
\begin{equation}
\label{eq:quantitative_closing_lemma_hamiltonian_diffeomorphisms_period_bound}
d\leq g + \lfloor Al^{-1}\delta^{-1}\rfloor.
\end{equation}
\end{theorem}

Standard arguments then imply the following result, which was first proved by Asaoka-Irie \cite[Corollary 1.2]{AI16}.

\begin{cor}
The periodic orbits of a $C^\infty$ generic Hamiltonian diffeomorphism of $(\Sigma,\omega)$ form a dense subset of $\Sigma$.
\end{cor}

\begin{remark}
Using PFH spectral invariants, one can show that Theorem \ref{theorem:quantitative_closing_lemma_hamiltonian_diffeomorphisms} continues to hold for arbitrary rational area-preserving diffeomorphisms $\phi$ (possibly with slightly different expressions for the bound on the period \eqref{eq:quantitative_closing_lemma_hamiltonian_diffeomorphisms_period_bound}). This is done in \cite{EH21} (in combination with the existence result for $U$-cyclic classes from \cite{CPPZ21}).
\end{remark}

\subsection{Simplicity conjecture and extension of Calabi}
\label{subsec:simplicity_conjecture_and_extension_of_calabi}

We summarize recent breakthroughs in $C^0$ symplectic geometry which we are able to reprove using our elementary spectral invariants. In the following, we deviate from our convention and allow the surface $(\Sigma,\omega)$ to have boundary. This means that $\Sigma$ is a compact, connected, orientable surface, possibly with boundary. We begin by recalling some basic concepts. We let $\on{Ham}(\Sigma,\omega)$ denote the group of Hamiltonian diffeomorphisms of $(\Sigma,\omega)$. If the boundary of $\Sigma$ is non-empty, the generating Hamiltonians are required to be compactly supported in the interior of $\Sigma$. We recall that $\on{Ham}(\Sigma,\omega)$ admits a bi-invariant metric $d_H$ called the {\it Hofer metric}. It is defined as follows: Given a Hamiltonian $H:[0,1]\times\Sigma\rightarrow\BR$, its {\it Hofer norm} $\|H\|$ is given by
\begin{equation}
\|H\|\coloneqq \int_0^1 \left( \max_\Sigma H_t - \min_\Sigma H_t \right) dt.
\end{equation}
The Hofer distance $d_H$ between $\phi,\psi\in\on{Ham}(\Sigma,\omega)$ is defined by
\begin{equation}
d_H(\phi,\psi)\coloneqq \inf\{\|H\|\mid \phi\circ\psi^{-1} = \phi_H^1\}.
\end{equation}
Next, let us introduce the group $\ol{\on{Ham}}(\Sigma,\omega)$ of {\it Hamiltonian homeomorphisms}. First, we define $\on{Homeo}_0(\Sigma,\omega)$ to be the component of the identity of the group of all homeomorphisms of $\Sigma$ which agree with the identity in some neighbourhood of the boundary and preserve the measure induced by the area form $\omega$. We define the metric $d_{C^0}$ on $\on{Homeo}_0(\Sigma,\omega)$ by
\begin{equation}
d_{C^0}(\phi,\psi)\coloneqq \sup_{x\in \Sigma} d(\phi(x),\psi(x))
\end{equation}
where $d$ is some auxiliary Riemannian distance on $\Sigma$. Then $\ol{\on{Ham}}(\Sigma,\omega)$ is defined to be the closure of $\on{Ham}(\Sigma,\omega)$ inside $\on{Homeo}_0(\Sigma,\omega)$ with respect to the $C^0$-metric $d_{C^0}$. It is a normal subgroup of $\on{Homeo}_0(\Sigma,\omega)$ and it is proper unless $\Sigma\in\{\BD,S^2\}$.\\

It has been a long-standing open question going back to Fathi \cite{Fat80} whether the group $\ol{\on{Ham}}(\Sigma,\omega)$ is simple. This problem is known as the {\it simplicity conjecture} (see e.g. \cite{MS17}). It was recently resolved by Cristofaro-Gardiner-Humili\`{e}re-Seyfaddini in \cite{CHS20} and \cite{CHS21} in the case $\Sigma\in\{\BD,S^2\}$ using PFH spectral invariants. For arbitrary $\Sigma$, the conjecture was settled by Cristofaro-Gardiner-Humili\`{e}re-Mak-Seyfaddini-Smith in \cite{CHMSS21} using spectral invariants defined via quantitative Heegaard Floer homology.

\begin{theorem}[Simplicity conjecture \cite{CHS20} \cite{CHS21} \cite{CHMSS21}]
\label{theorem:simplicity_conjecture}
The group $\ol{\on{Ham}}(\Sigma,\omega)$ is not simple.
\end{theorem}

We are able to reprove this theorem using our elementary spectral invariants. We follow the same basic strategy as \cite{CHS20}, \cite{CHS21} and \cite{CHMSS21}.

\begin{definition}
\label{def:finite_energy_homeomorphisms}
A homeomorphism $\phi\in\ol{\on{Ham}}(\Sigma,\omega)$ is called a {\it finite energy homeomorphism} if there exists a sequence of Hamiltonians $H_j\in C^\infty([0,1]\times\Sigma)$ such that the sequence of Hofer norms $\|H_j\|$ is bounded and such that $\phi_{H_j}^1$ converges to $\phi$ with respect to the $C^0$-metric $d_{C^0}$. We let $\on{FHomeo}(\Sigma,\omega)$ denote the group of finite energy homeomorphisms.
\end{definition}

The finite energy homeomorphisms $\on{FHomeo}(\Sigma,\omega)$ form a normal subgroup of $\ol{\on{Ham}}(\Sigma,\omega)$ (see \cite{CHS20}). Using the Weyl law for the spectral invariants $c_d$, we will show that $\on{FHomeo}(\Sigma,\omega)$ is in fact a proper subgroup. This implies Theorem \ref{theorem:simplicity_conjecture}.\\

The second application which we are able to reprove using our spectral invariants involves the {\it Calabi homomorphism}. Assume that $\Sigma$ has non-empty boundary. Then the Calabi homomorphism $\on{Cal}:\on{Ham}(\Sigma,\omega)\rightarrow\BR$ is defined as follows. Given $\phi\in\on{Ham}(\Sigma,\omega)$, choose $H:[0,1]\times\Sigma\rightarrow \BR$ compactly supported in the interior of $\Sigma$ such that $\phi = \phi_H^1$. Then
\begin{equation}
\on{Cal}(\phi) = \int_{[0,1]\times\Sigma} H dt\wedge\omega
\end{equation}
which turns out to be independent of the choice of Hamiltonian $H$ generating $\phi$. Fathi conjectured that $\on{Cal}$ extends to a homomorphism on the group $\on{Hameo}(\Sigma,\omega)$ of {\it Hameomorphisms} introduced by M\"{u}ller-Oh in \cite{OM07}.

\begin{definition}
\label{def:hameomorphisms}
A homeomorphism $\phi\in\ol{\on{Ham}}(\Sigma,\omega)$ is called a {\it Hameomorphism} if there exists a sequence of Hamiltonians $H_j\in C^\infty([0,1]\times\Sigma)$ and a continuous Hamiltonian $H\in C^0([0,1]\times\Sigma)$ such that the sequence of Hofer norms $\|H_j-H\|$ converges to zero and such that $\phi_{H_j}^1$ converges to $\phi$ with respect to the $C^0$-metric $d_{C^0}$. We let $\on{Hameo}(\Sigma,\omega)$ denote the group of Hameomorphisms.
\end{definition}

\begin{theorem}[Extension of Calabi \cite{CHMSS21}]
\label{theorem:calabi_extension}
Suppose that $\Sigma$ has non-empty boundary. The Calabi homomorphism extends to a homomorphism $\on{Cal}:\on{Hameo}(\Sigma,\omega)\rightarrow\BR$.
\end{theorem}

This result was proved in \cite{CHMSS21} using quantitative Heegaard Floer homology. We reprove it using our elementary spectral invariants. Again, the Weyl law for the spectral invariants $c_d$ plays a fundamental role.\\

\noindent{\bf Organization.} The rest of the paper is structured as follows:\\
\indent In \S\ref{sec:elemantary_construction_of_spectral_invariants} we give a detailed construction of our elementary spectral invariants and prove their basic properties summarized in Theorems \ref{theorem:spectral_invariants_properties_basic} and \ref{theorem:spectral_invariants_properties_weyl}. In \S\ref{subsec:almost_complex_structures} we define the class of almost complex structures used in our construction. Section \S\ref{subsec:pseudo_holomorphic_curves_and_energy} introduces the relevant notion of energy of pseudo-holomorphic curves. The definition of $c_{d,k}$ along with some elementary properties is given in \S\ref{subsec:definition_of_c_dk}. The neck stretching arguments used in order to prove monotonicity, the packing property and the closed curve upper bound are explained in \S\ref{subsec:neckstretching}. The proofs of Theorems \ref{theorem:spectral_invariants_properties_basic} and \ref{theorem:spectral_invariants_properties_weyl} are completed in \S\ref{subsec:proofs_of_basic_properties}.\\
\indent Section \S\ref{sec:applications} is devoted to applications of our elementary spectral invariants. In \S\ref{subsec:quantitative_closing_lemma} we treat quantitative closing lemmas. A proof of the simplicity conjecture and the extension of the Calabi homomorphism are the subject of \S\ref{subsec:simplicity_conjecture_and_calabi}.\\
\indent In the appendix we prove some lemmas concerning the relative action spectrum. Moreover, we provide some computations of certain Gromov-Taubes invariants, both with and without the use of Seiberg-Witten theory.\\

\noindent{\bf Acknowledgments.} We thank Michael Hutchings for suggesting this project and for helpful discussions. We acknowledge support from a UC Berkeley Department of Mathematics Spring Fellowship.

\section{Elementary construction of spectral invariants}
\label{sec:elemantary_construction_of_spectral_invariants}

The goal of this section is to construct the spectral invariants $c_{d,k}$ and $c_d$ and prove the properties summarized in Theorems \ref{theorem:spectral_invariants_properties_basic} and \ref{theorem:spectral_invariants_properties_weyl}.

\subsection{Almost complex structures}
\label{subsec:almost_complex_structures}

Fix an area-preserving diffeomorphism $\phi$ of $(\Sigma,\omega)$. Our spectral invariants are based on pseudo-holomorphic curves in the symplectization $(M_\phi,\Omega_\phi)$ defined in \eqref{eq:symplectization_definition}. In this subsection, we specify the almost complex structures on $(M_\phi,\Omega_\phi)$ relevant to our construction.

\begin{definition}
\label{def:symplectization_admissible_almost_complex_structure}
We call an almost complex structure $J$ on $M_\phi$ {\it symplectization admissible} if it satisfies the following properties:
\begin{enumerate}
\item $J$ is compatible with $\Omega_\phi$.
\item $J$ preserves the vertical tangent bundle of $M_\phi$, i.e.\ the subbundle of $TM_\phi$ consisting of all vectors tangent to the fibres of the projection $M_\phi\rightarrow \BR\times\BR/\BZ$.
\item $J\partial_s = \partial_t$
\item $J$ is invariant under translations along the symplectization direction $\partial_s$.
\end{enumerate}
\end{definition}

\begin{definition}
\label{def:adapted_almost_complex_structure}
Let $J$ be an almost complex structure defined on some subset of $M_\phi$ and let $H\in C^\infty(Y_\phi)$ be a Hamiltonian. We call $J$ {\it adapted} to $H$ if the pull-back $f_H^*J$ under the symplectomorphism $f_H:M_{\phi_H}\rightarrow M_\phi$ defined in \eqref{eq:diffeomorphism_between_symplectizations} is the restriction of a symplectization admissible almost complex structure on $M_{\phi_H}$.
\end{definition}

\begin{definition}
\label{def:almost_complex_structures}
Given Hamiltonians $H_\pm\in C^\infty(Y_\phi)$, we define $\MJ(H_+,H_-)$ to be the set of all almost complex structures on $M_\phi$ which are compatible with $\Omega_\phi$ and adapted to $H_+$ and $H_-$ on the subsets $[s,\infty)\times Y_\phi$ and $(-\infty,-s]\times Y_\phi$, respectively, for $s\gg 0$ sufficiently large. Here $s$ is allowed to depend on $J$.
\end{definition}

Suppose now that $H_\pm$ are Hamiltonians satisfying $H_+>H_-$. Recall that in this situation $M_{H_+,H_-}$, i.e.\ the closure of the region in $M_\phi$ between the graphs $\on{gr}(H_\pm)$, can be regarded as a symplectic cobordism between the mapping tori $Y_{\phi_{H_\pm}}$. Moreover, recall that $\on{gr}^\pm(H_\pm)$ can be viewed as a positive/negative cylindrical end attached to the cobordism $M_{H_+,H_-}$. Let $J\in\MJ(H_+,H_-)$. It is a direct consequence of Definition \ref{def:almost_complex_structures} that the restriction of $J$ to each of the cylindrical ends is symplectization admissible outside some compact subset. For technical reasons, it is important to not require $J$ to be symplectization admissible on the entire cylindrical end. Let us stress that the definition of $\MJ(H_+,H_-)$ also makes sense if $H_+>H_-$ does not hold and the cobordism $M_{H_+,H_-}$ is not defined. Moreover, we observe that $\MJ(H_+,H_-)$ does not change if we shift the Hamiltonians $H_\pm$ by adding constants.

\subsection{Pseudo-holomorphic curves and energy}
\label{subsec:pseudo_holomorphic_curves_and_energy}

Let us call a Hamiltonian $H\in C^\infty(Y_\phi)$ {\it non-degenerate} if all periodic orbits of $\phi_H$ are non-degenerate. The set of all non-degenerate Hamiltonians is residual in $C^\infty(Y_\phi)$ with respect to the $C^\infty$ topology. Suppose that $H_\pm$ are non-degenerate Hamiltonians and let $J\in\MJ(H_+,H_-)$. Our definition of spectral numbers is based on $J$-holomorphic maps $u:(S,j)\rightarrow (M_\phi,J)$. Here $(S,j)$ is a possibly disconnected Riemann surface obtained from a compact Riemann surface by removing finitely many punctures. We require that at every puncture of $S$ either $f_{H_+}^{-1}\circ u$ is positively asymptotic to the cylinder $\BR\times\gamma_+$ for a periodic orbit $\gamma_+$ of $Y_{\phi_{H_+}}$ or $f_{H_-}^{-1}\circ u$ is negatively asymptotic to the cylinder $\BR\times \gamma_-$ for a periodic orbit $\gamma_-$ of $Y_{\phi_{H_-}}$. Moreover, we require that $u$ is non-constant on all connected components of $S$. We let $\MM^J(H_+,H_-)$ denote the space of all such $J$-holomorphic maps modulo reparametrization of the domain. If $H_+>H_-$, then the symplectic cobordism $M_{H_+,H_-}$ is defined and $\MM^J(H_+,H_-)$ is simply the set of $J$-holomorphic curves in the completion of $M_{H_+,H_-}$ positively and negatively asymptotic to periodic orbits. However, the set $\MM^J(H_+,H_-)$ is defined for any pair of (non-degenerate) Hamiltonians. Moreover, it does not change if we shift the Hamiltonians by adding constants. Given $u\in\MM^J(H_+,H_-)$, we define its {\it degree} to be the intersection number with a fibre of the projection $M_\phi\rightarrow \BR\times\BR/\BZ$. For every positive integer $d\geq 1$, we define $\MM^J(H_+,H_-,d)$ to be the subset consisting of all curves of degree $d$. If $P\subset M_\phi$ is a finite subset, we further define $\MM^J(H_+,H_-,d,P)$ to be the subset of curves containing $P$ in their image.\\
Next, let us define the {\it energy} $\ME_{H_+,H_-}(u)$ of a $J$-holomorphic curve $u\in\MM^J(H_+,H_-,d)$. This energy takes values in the relative action spectrum $\on{Spec}_d(H_+,H_-)$. Let $\alpha_\pm$ be the orbit sets of $Y_{\phi_{H_\pm}}$ such that $u$ is positively asymptotic to $\alpha_+$ and negatively asymptotic to $\alpha_-$. As in the definition of the relative action spectrum (Definition \ref{def:relative_action_spectrum}), we can view $\alpha_\pm$ as $1$-chains in $M_\phi$ via the identification of $Y_{\phi_{H_\pm}}$ with $\on{gr}(H_\pm)$. The curve $u$ induces a relative homology class $[u]\in H_2(M_\phi,\alpha_+,\alpha_-)$ as follows. We form a $2$-chain in the compactification $\ol{M_\phi}=(\BR\cup\{\pm\infty\})\times Y_\phi$ by concatenating the cylinders $f_{H_+}([0,+\infty]\times\alpha_+)$ and $f_{H_-}([-\infty,0]\times\alpha_-)$ with the pseudo-holomorphic curve $u$. The boundary of this $2$-chain is given by $\alpha_+-\alpha_-$ and we obtain a relative homology class $[u]$ in $H_2(\ol{M_\phi},\alpha_+,\alpha_-)\cong H_2(M_\phi,\alpha_+,\alpha_-)$.

\begin{definition}
\label{def:engergy}
We define the {\it energy} of a pseudo-holomorphic curve $u\in\MM^J(H_+,H_-,d)$ to be
\begin{equation}
\ME_{H_+,H_-}(u)\coloneqq \langle [u],\Omega_\phi \rangle.
\end{equation}
\end{definition}

Clearly, $\ME_{H_+,H_-}(u)\in \on{Spec}_d(H_+,H_-)$. While the set of almost complex structures $\MJ(H_+,H_-)$ and the set of $J$-curves $\MM^J(H_+,H_-)$ only depend on the Hamiltonians $H_\pm$ up to additive constants, the relative action spectrum $\on{Spec}_d(H_+,H_-)$ and energy $\ME_{H_+,H_-}$ behave as follows under shifting the Hamiltonians.

\begin{lem}
\label{lem:shift_spectrum_energy}
Let $H_\pm \in C^\infty(Y_\phi)$ be Hamiltonians and let $C_\pm\in\BR$ be real numbers. We define Hamiltonians $H_\pm'\coloneqq C_\pm+H_\pm$. Then
\begin{equation}
\label{eq:shift_spectrum_energy_a}
\on{Spec}_d(H_+',H_-') = \on{Spec}_d(H_+,H_-) + d(C_+-C_-).
\end{equation}
If $H_\pm$ are non-degenerate and $u\in\MM^J(H_+,H_-,d)=\MM^J(H_+',H_-',d)$ is a $J$-curve of degree $d$, then
\begin{equation}
\label{eq:shift_spectrum_energy_b}
\ME_{H_+',H_-'}(u) = \ME_{H_+,H_-}(u) + d(C_+-C_-).
\end{equation}
\end{lem}

\begin{remark}
The energy of a $J$-curve can be negative, which is immediate from \eqref{eq:shift_spectrum_energy_b}.
\end{remark}

\begin{proof}[Proof of Lemma \ref{lem:shift_spectrum_energy}]
Let $\alpha_\pm$ be orbit sets in  the mapping tori $Y_{\phi_{H_\pm}} = Y_{\phi_{H_\pm'}}$. We may view $\alpha_\pm$ as $1$-cycles inside $\on{gr}(H_\pm)$ via the identification $Y_{\phi_{H_\pm}}\cong \on{gr}(H_\pm)$ or as $1$-cycles inside $\on{gr}(H_\pm')$ via $Y_{\phi_{H_\pm'}}\cong \on{gr}(H_\pm')$. We use the notation $\alpha_\pm$ to refer to the former $1$-cycles and $\alpha_\pm'$ to refer to the latter. Let $Z_\pm\in H_2(M_\phi,\alpha_\pm',\alpha_\pm)$ denote the relative homology classes represented by the union of cylinders of height $C_\pm$ connecting the periodic orbits contained in $\alpha_\pm'$ and $\alpha_\pm$. We have a bijection between relative homologies given by
\begin{equation}
H_2(M_\phi,\alpha_+,\alpha_-)\longrightarrow H_2(M_\phi,\alpha_+',\alpha_-')\qquad Z\mapsto Z + Z_+ - Z_-.
\end{equation}
Equations \eqref{eq:shift_spectrum_energy_a} and \eqref{eq:shift_spectrum_energy_b} follow from the observation that the $\Omega_\phi$-areas of $Z_\pm$ are given by $\langle Z_\pm,\Omega_\phi\rangle = dC_\pm$.
\end{proof}

\begin{lem}
\label{lem:lower_energy_bound}
For every tuple of non-degenerate Hamiltonians $H_\pm$, every $J\in\MJ(H_+,H_-)$ and every positive integer $d\geq 1$, there exists a constant $C\in\BR$ such that $\ME_{H_+,H_-}(u) \geq C$ for all $u\in\MM^J(H_+,H_-,d)$.
\end{lem}

\begin{proof}
By Lemma \ref{lem:shift_spectrum_energy} it suffices to prove the assertion for suitable shifts of the Hamiltonians $H_\pm$. We may therefore assume w.l.o.g. that $H_+>H_-$ and that the almost complex structure is symplectization admissible on the entirety of the cylindrical ends $\on{gr}^\pm(H_\pm)$. We claim that under this assumption the energy $\ME_{H_+,H_-}(u)$ is non-negative for every pseudo-holomorphic curve $u\in\MM^J(H_+,H_-,d)$. In order to see this, we split $u$ into three pieces:
\begin{equation}
\label{eq:lower_energy_bound_proof_a}
u_\pm \coloneqq u \cap \on{gr}^\pm(H_\pm) \qquad \text{and} \qquad u_0 \coloneqq u \cap M_{H_+,H_-}
\end{equation}
The energy $\ME_{H_+,H_-}(u)$ can be expressed as
\begin{equation}
\label{eq:lower_energy_bound_proof_b}
\ME_{H_+,H_-}(u) \coloneqq \int_{f_{H_+}^{-1}u_+} \omega_{\phi_{H_+}} + \int_{u_0}\Omega_\phi + \int_{f_{H_-}^{-1}u_-} \omega_{\phi_{H_-}}.
\end{equation}
The second summand is non-negative because $J$ is compatible with $\Omega_\phi$. The first and third summands are non-negative because $J$ is symplectization admissible on the cylindrical ends $\on{gr}^\pm(H_\pm)$.
\end{proof}

\subsection{Definition of $c_{d,k}$}
\label{subsec:definition_of_c_dk}

We are now ready to define the spectral invariants $c_{d,k}$. At first, we only define them for non-degenerate Hamiltonians $H_\pm$. After proving some of their basic properties, in particular Lipschitz continuity, we will extend them to arbitrary Hamiltonians.

\begin{definition}
\label{def:spectral_invariants_basic}
For $d \geq 1$ and $k \geq 0$ and non-degenerate Hamiltonians $H_\pm$, we define
\begin{equation}
c_{d,k}(H_+,H_-) \coloneqq \sup\limits_{\substack{J\in\MJ(H_+,H_-)\\ |P|=k}} \enspace \inf\limits_{u\in \MM^J(H_+,H_-,d,P)} \enspace\ME_{H_+,H_-}(u).
\end{equation}
\end{definition}
By Lemma \ref{lem:lower_energy_bound}, the infimum of $\ME_{H_+,H_-}(u)$ over all $u\in \MM^J(H_+,H_-,d,P)$ takes values in $\BR\cup\{+\infty\}$. Thus $c_{d,k}(H_+,H_-)$ takes values in $\BR\cup\{+\infty\}$ as well. The following basic property of $c_{d,k}$ is an immediate consequence of Lemma \ref{lem:shift_spectrum_energy}.

\begin{lem}[Shift]
\label{lem:shift_spectral_number}
Let $H_\pm \in C^\infty(Y_\phi)$ be non-degenerate Hamiltonians and let $C_\pm\in\BR$ be real numbers. We define Hamiltonians $H_\pm'\coloneqq C_\pm+H_\pm$. For every $d\geq 1$ and $k\geq 0$, we have
\begin{equation}
c_{d,k}(H_+',H_-') = c_{d,k}(H_+,H_-) + d(C_+-C_-).
\end{equation}
\end{lem}

\begin{proof}
This is an immediate consequence of Lemma \ref{lem:shift_spectrum_energy} since neither $\MJ(H_+,H_-)$ nor the moduli space $\MM^J(H_+,H_-,d,P)$ are affected by shifting the Hamiltonians.
\end{proof}

\begin{lem}[Non-negativity]
\label{lem:non_negativity}
Let $H_\pm$ be non-degenerate Hamiltonians and assume that $H_+\geq H_-$. Then $c_{d,k}(H_+,H_-)$ is non-negative for all $d\geq 1$ and all $k\geq 0$.
\end{lem}

\begin{proof}
By Lemma \ref{lem:shift_spectral_number}, we may shift $H_+$ by an arbitrarily small positive amount and assume w.l.o.g. that $H_+ > H_-$. Choose an $\Omega_\phi$-compatible almost complex structure $J$ on $M_\phi$ which is adapted to the Hamiltonians $H_\pm$ on the entirety of the cylindrical ends $\on{gr}^\pm(H_\pm)$. Moreover, choose a set $P\subset M_\phi$ of cardinality $k$. Since the definition of $c_{d,k}$ takes the supremum over all tuples $(J,P)$, if suffices to show that for this specific choice of $(J,P)$ the energy $\ME_{H_+,H_-}(u)$ is non-negative for all $u\in\MM^J(H_+,H_-,d,P)$. This is done exactly as in the proof of Lemma \ref{lem:lower_energy_bound}: We split $u$ into three pieces as in \eqref{eq:lower_energy_bound_proof_a} and observe that all three summands in \eqref{eq:lower_energy_bound_proof_b} must be non-negative.
\end{proof}

\subsection{Neck stretching}
\label{subsec:neckstretching}

In this subsection we prove that the spectral invariants $c_{d,k}$ satisfy monotonicity, the packing property and the closed curve upper bound in Theorem \ref{theorem:spectral_invariants_properties_basic}. This is done by adapting the neck stretching argument in \cite{Hut22} to our current setting.

\begin{lem}[Monotonicity]
\label{lem:monotonicity}
Let $d\geq 1$ be a positive integer. Let $k_\pm\geq 0$ be non-negative integers and set $k\coloneqq k_++k_-$. Let $H_\pm, H_0\in C^\infty(Y_\phi)$ be non-degenerate Hamiltonians. Then
\begin{equation}
\label{eq:monotonicity}
c_{d,k}(H_+,H_-)\geq c_{d,k_+}(H_+,H_0) + c_{d,k_-}(H_0,H_-).
\end{equation}
\end{lem}

\begin{proof}
By the shift property (Lemma \ref{lem:shift_spectral_number}), it suffices to show the inequality
\begin{equation}
c_{d,k}(H_+,H_-)\geq c_{d,k_+}(H_+-\delta,H_0+\delta) + c_{d,k_-}(H_0-\delta,H_-+\delta)
\end{equation}
for all $\delta>0$. Let us fix $\delta>0$ and abbreviate
\begin{equation}
H_+^-\coloneqq H_+-\delta\qquad H_0^+\coloneqq H_0+\delta \qquad H_0^- \coloneqq H_0 - \delta \qquad H_-^+\coloneqq H_-+\delta.
\end{equation}
Fix almost complex structures $J_+\in\MJ(H_+^-,H_0^+)$ and $J_-\in \MJ(H_0^-,H_-^+)$. Moreover, fix sets $P_\pm\subset M_\phi$ of cardinalities $k_\pm$. Let $\epsilon>0$ be arbitrary. Our goal is to construct pseudoholomorphic curves $u_+\in \MM^{J_+}(H_+^-,H_0^+,d,P_+)$ and $u_- \in \MM^{J_-}(H_0^-,H_-^+,d,P_-)$ such that
\begin{equation}
\label{eq:monotonicity_proof_a}
c_{d,k}(H_+,H_-) + \epsilon \geq \ME_{H_+^-,H_0^+}(u_+) + \ME_{H_0^-,H_-^+}(u_-).
\end{equation}
We closely follow the neck-stretching argument explained in \cite[Section 3]{Hut22} to find such curves $u_\pm$. Since $J_\pm$, $P_\pm$ and $\epsilon >0$ are arbitrary, inequality \eqref{eq:monotonicity_proof_a} implies the desired inequality \eqref{eq:monotonicity}.\\
We claim that we may assume w.l.o.g. that the following is true:
\begin{enumerate}
\item $H_+^->H_0^+$ and $H_0^->H_-^+$
\item $P_+$ is contained in the interior of $M_{H_+^-,H_0^+}$ and $P_-$ is contained in the interior of $M_{H_0^-,H_-^+}$.
\item $J_+$ is symplectization admissible on the entire cylindrical ends attached to $M_{H_+^-,H_0^+}$. Similarly, $J_-$ is symplectization admissible on the entire cylindrical ends attached to $M_{H_0^-,H_-^+}$.
\end{enumerate}
Indeed, by Lemmas \ref{lem:shift_spectrum_energy} and \ref{lem:shift_spectral_number}, shifting the Hamiltonians $H_\pm$ and $H_0$ changes both sides of \eqref{eq:monotonicity_proof_a} by the same amount. Moreover, if we shift $J_+$ and $P_+$ along the symplectization direction, the moduli space $\MM^{J_+}(H_+^-,H_0^+,d,P_+)$ also changes by a shift. The energy $\ME_{H_+^-,H_0^+}$ is not affected by shifting holomorphic curves. The analogous statement holds for $J_-$ and $P_-$. Properties (1)-(3) above can be achieved by suitable shifts of the Hamiltonians, the almost complex structures $J_\pm$ and the points $P_\pm$.\\
For every $R>0$, we may choose a compatible almost complex structure $J^R$ on $M_\phi$ and biholomorphisms
\begin{equation}
\psi_+^R : (M_{H_+^-+R,H_0^+-R},J_+) \rightarrow (M_{H_+^-+\delta/2,H_0^+-\delta/2},J^R)
\end{equation}
and
\begin{equation}
\psi_-^R : (M_{H_0^-+R,H_-^+-R},J_-) \rightarrow (M_{H_0^-+\delta/2,H_-^+-\delta/2},J^R)
\end{equation}
with the following properties:
\begin{enumerate}
\item $J^R$ is independent of $R$ on the complement of $M_{H_+^-+\delta/2,H_0^+-\delta/2}\cup M_{H_0^-+\delta/2,H_-^+-\delta/2}$
\item $J^R$ is adapted to $H_+$ and $H_-$ on the sets $\on{gr}^+(H_+)$ and $\on{gr}^-(H_-)$, respectively.
\item The restriction of $\psi_+^R$ to $M_{H_+^-,H_0^+}$ is the identity. Similarly, the restriction of $\psi_-^R$ to $M_{H_0^-,H_-^+}$ is the identity.
\item For every $s\in [0,R]$, the restriction of $\psi_+^R$ to $\on{gr}(H_+^-+s)$ or $\on{gr}(H_0^+-s)$ is a shift in the symplectization direction by some amount depending on $s$. An analogous statement holds for the restriction of $\psi_-^R$ to $\on{gr}(H_0^-+s)$ or $\on{gr}(H_-^+-s)$.
\end{enumerate}
By definition of $c_{d,k}$, we may choose a $J^R$-curve $u^R\in \MM^{J^R}(H_+,H_-,d,P_+\cup P_-)$ such that
\begin{equation}
\label{eq:monotonicity_proof_b}
c_{d,k}(H_+,H_-)+\epsilon > \ME_{H_+,H_-}(u^R).
\end{equation}
Since we do not have an apriori bound on the genus of the curves $u^R$, we cannot directly use SFT compactness (see \cite{BEHWZ03}). As explained in \cite{Hut22}, there is a version of Gromov compactness due to Taubes that does not require genus bounds (see \cite[Proposition 3.3]{Tau98}). It follows from this compactness result that there exists a sequence $R_i \rightarrow +\infty$ such that the curves $(\psi_\pm^{R_i})^{-1}\circ u^{R_i}$ converge to proper pseudo-holomorphic maps $u_\pm$ in $(M_\phi,J_\pm)$. The maps $u_\pm$ must pass through the points $P_\pm$ and their degree is equal to $d$. As a current, $f_{H_+^-}^{-1}\circ u_+$ is positively asymptotic to a degree $d$ orbit set of $Y_{\phi_{H_+^-}}$ and $f_{H_0^+}^{-1}\circ u_+$ is negatively asymptotic to a degree $d$ orbit set of $Y_{\phi_{H_0^+}}$. The analogous statement is true for $u_-$. It follows from \eqref{eq:monotonicity_proof_b} that $u_\pm$ satisfy \eqref{eq:monotonicity_proof_a}. It remains to argue that $u_\pm$ can be arranged to have finite genus.\\
The curves $u^{R_i}$ may have multiply covered components. Let $v^{R_i}$ denote the curve obtained from $u^{R_i}$ by replacing every multiply covered component with the underlying simple curve. Since this operation does not increase the energy $\ME_{H_+,H_-}$, we still have a uniform bound on $\ME_{H_+,H_-}(v^{R_i})$. Moreover, the degree of $v^{R_i}$ is at most $d$. Since $H_\pm$ are non-degenerate, there are only finitely many orbit sets in $Y_{\phi_{H_\pm}}$ of degree at most $d$. After possibly passing to a further subsequence, we may assume that there exist orbit sets $\alpha_\pm$ of degree at most $d$ such that all the curves $v^{R_i}$ are positively asymptotic to $\alpha_+$ and negatively asymptotic to $\alpha_-$. As explained in \cite[Section 9.4]{Hut02}, we can use Gromov compactness to pass to a subsequence such that all the curves $v^{R_i}$ represent the same relative homology class in $H_2(M_\phi,\alpha_+,\alpha_-)$. It follows from the relative adjunction formula and the asymptotic writhe bound (see \cite[Proposition 4.9 and Lemma 4.20]{Hut09}) that there is a uniform lower bound on the Euler characteristic of the curves $v^{R_i}$. This implies the desired uniform bound on the genus of the curves $u^{R_i}$. It follows that the limit curves $u_\pm$ have finite genus and are indeed contained in $\MM^{J_\pm}(H_+,H_-,d,P_\pm)$.
\end{proof}

\begin{lem}[Packing]
\label{lem:packing}
Let $d\geq 1$ be a positive integer and $k\geq 0$ a non-negative integer. Let $H_\pm\in C^\infty(Y_\phi)$ be non-degenerate Hamiltonians and assume that $H_+>H_-$. Let $X$ denote a compact Liouville domain which admits a symplectic embedding into the symplectic cobordism $M_{H_+,H_-}$. Then
\begin{equation}
\label{eq:packing}
c_{d,k}(H_+,H_-) \geq c_k^H(X).
\end{equation}
\end{lem}

\begin{proof}
The proof of the monotonicity lemma given in \cite[Section 3]{Hut22} carries over verbatim to show the inequality \eqref{eq:packing}.
\end{proof}

\begin{lem}[Closed curve upper bound]
\label{lem:closed_curve_upper_bound}
Let $(X,\Omega)$ be a closed symplectic $4$-manifold. Assume that $H_\pm\in C^\infty(Y_\phi)$ are non-degenerate Hamiltonians satisfying $H_+>H_-$. Suppose that there exists a symplectic embedding $\iota:M_{H_+,H_-} \rightarrow X$. Let $Z\in H_2(X;\BZ)$ be a homology class and assume that the Gromov-Taubes invariant $\on{Gr}(X,\Omega;Z)$ does not vanish (see appendix \ref{subsec:review_taubes_gromov_invariants} for a review of $\on{Gr}$). Let $k\geq 0$ and $d\geq 1$ be integers such that $I(Z)=2k$ and $Z\cdot \iota_*[\Sigma]=d$. Then
\begin{equation}
\label{eq:closed_curve_upper_bound}
c_{d,k}(H_+,H_-)\leq \langle Z,[\Omega] \rangle.
\end{equation}
\end{lem}

\begin{proof}
The proof of this lemma is very similar to the proof of Lemma \ref{lem:monotonicity}. We fix an almost complex structure $J\in\MJ(H_+,H_-)$ and a set $P\subset M_\phi$ of cardinality $k$. Let $\epsilon>0$. The goal is to construct a pseudo-holomorphic curve $u\in\MM^J(H_+,H_-,d,P)$ such that
\begin{equation}
\label{eq:closed_curve_upper_bound_proof_a}
\ME_{H_+,H_-}(u) \leq \langle Z,[\Omega] \rangle + \epsilon.
\end{equation}
The desired inequality \eqref{eq:closed_curve_upper_bound} follows since $\epsilon$ can be chosen arbitrarily small.\\
In order to find such a curve $u$ via neck stretching, we need the following additional assumption on $J$:
\begin{itemize}
\item[($\bullet$)] $J$ is symplectization admissible everywhere on the cylindrical ends attached to the cobordism $M_{H_+,H_-}$.
\end{itemize}
We may reduce to this case as follows: Let $H_\pm'$ be shifts of $H_\pm$ with the property the the almost complex structure $J$ is symplectization admissible on the entire cylindrical ends of the cobordism $M_{H_+',H_-'}$. Let $C_\pm$ be constants such that $H_\pm' = H_\pm + C_\pm$. We cut the cobordism $M_{H_+,H_-}$ out of $X$ and insert the cobordism $M_{H_+',H_-'}$ instead. This yields a new symplectic manifold $(X',\Omega')$. Clearly, the manifolds $X$ and $X'$ are diffeomorphic and we may regard $\Omega'$ as a symplectic form on $X$ which can be connected to $\Omega$ through a path of symplectic forms. The Gromov-Taubes invariant is invariant under such deformations (see \cite{Tau96}). Hence $\on{Gr}(X,\Omega';Z)$ is non-zero. We have
\begin{equation}
\langle Z,[\Omega'] \rangle = \langle Z,[\Omega] \rangle + d(C_+-C_-).
\end{equation}
Using Lemma \ref{lem:shift_spectrum_energy}, we see that replacing $H_\pm$ by $H_\pm'$ and $\Omega$ by $\Omega'$ changes both sides of \eqref{eq:closed_curve_upper_bound_proof_a} by the same amount. Since the additional assumption ($\bullet$) holds for $H_\pm'$, this concludes our reduction to this case.\\
The assumption that $\on{Gr}(X,\Omega;Z)$ does not vanish implies that for any compatible almost complex structure on $(X,\Omega)$ and any collection of $k$ distinct points in $X$ there exists a pseudo-holomorphic curve in $X$ through these points and representing the homology class $Z$ (see Lemma \ref{lem:gromov_taubes_nonvanishing_implies_curves}). We may neck stretch these curves in the exact same way as in the proof of Lemma \ref{lem:monotonicity} to find a curve $u$ satisfying \eqref{eq:closed_curve_upper_bound_proof_a}. We omit the details.
\end{proof}

\subsection{Proofs of basic properties}
\label{subsec:proofs_of_basic_properties}

In this subsection we extend the spectral invariants $c_{d,k}$ to arbitrary Hamiltonians and prove Theorems \ref{theorem:spectral_invariants_properties_basic} and \ref{theorem:spectral_invariants_properties_weyl} on the basic properties of the spectral invariants $c_{d,k}$ and $c_d$. Moreover, we prove Theorem \ref{theorem:finiteness_for_hamiltonian_diffeomorphisms} on the finiteness of the spectral invariants in the case that $\phi$ is a Hamiltonian diffeomorphism. First, we state and prove some immediate corollaries of the monotonicity property (Lemma \ref{lem:monotonicity}).

\begin{cor}
\label{cor:monotonicity}
Let $H_\pm$ and $H_\pm'$ be non-degenerate Hamiltonians. Assume that $H_+\geq H_+'$ and $H_-\leq H_-'$. Then
\begin{equation}
c_{d,k}(H_+,H_-)\geq c_{d,k}(H_+',H_-')
\end{equation}
for all $d\geq 1$ and all $k\geq 0$.
\end{cor}

\begin{proof}
We use Lemmas \ref{lem:non_negativity} and \ref{lem:monotonicity} repeatedly to estimate
\begin{IEEEeqnarray}{rCl}
c_{d,k}(H_+,H_-) & \geq & c_{d,k}(H_+,H_-') + c_{d,0}(H_-',H_-) \nonumber\\
& \geq & c_{d,0}(H_+,H_+') + c_{d,k}(H_+',H_-') + c_{d,0}(H_-',H_-) \nonumber\\
& \geq & c_{d,k}(H_+',H_-'). \nonumber
\end{IEEEeqnarray}
\end{proof}

\begin{cor}
\label{cor:finite_or_infinite}
Fix $d\geq 1$ and $k\geq 0$. Then $c_{d,k}$ either takes values in $\BR$ for all pairs of non-degenerate Hamiltonians or is equal to $+\infty$ for all such pairs.
\end{cor}

\begin{proof}
Suppose that $c_{d,k}$ is finite for some pair of non-degenerate Hamiltonians $H_\pm$. By Lemma \ref{lem:shift_spectral_number}, $c_{d,k}$ must also be finite for all shifts of the Hamiltonians $H_\pm$. Together with Corollary \ref{cor:monotonicity} this implies that $c_{d,k}$ must be finite for all pairs of non-degenerate Hamiltonians.
\end{proof}

\begin{cor}
\label{cor:lipschitz_continuity}
Let $d\geq 1$ and $k\geq 0$. For $j\in \{0,1\}$ let $H_\pm^j\in C^\infty(Y_\phi)$ be non-degenerate Hamiltonians. Then
\begin{eqnarray*}
c_{d,k}(H_+^0,H_-^0) + d\cdot \left( \min_{Y_\phi} (H_+^1-H_+^0) + \min_{Y_\phi} (H_-^0-H_-^1)\right)
\leq c_{d,k}(H_+^1,H_-^1) \leq \\
c_{d,k}(H_+^0,H_-^0) + d\cdot \left( \max_{Y_\phi} (H_+^1-H_+^0) + \max_{Y_\phi} (H_-^0-H_-^1)\right).
\end{eqnarray*}
In particular, $c_{d,k}$ is Lipschitz continuous with Lipschitz constant $d$ in both entries. Here we endow $C^\infty(Y_\phi)$ with the $C^0$-norm.
\end{cor}

\begin{proof}
We observe that
\begin{equation}
\label{eq:cor_lipschitz_continuity_proof_a}
H_+^0 + \min_{Y_\phi} (H_+^1-H_+^0) \leq H_+^1 \quad \text{and}\quad H_-^0 - \min_{Y_\phi} (H_-^0-H_-^1) \geq H_-^1.
\end{equation}
The first inequality in the corollary follows from \eqref{eq:cor_lipschitz_continuity_proof_a}, Lemma \ref{lem:shift_spectral_number} and Corollary \ref{cor:monotonicity}. The second inequality follows from an analogous argument. Lipschitz continuity with Lipschitz constant $d$ is immediate from these inequalities.
\end{proof}

By Corollary \ref{cor:lipschitz_continuity} there exists a unique continuous extension
\begin{equation}
c_{d,k}: C^\infty(Y_\phi)\times C^\infty(Y_\phi)\rightarrow \BR\cup\{+\infty\}
\end{equation}
to all pairs of Hamiltonians $(H_+,H_-)$ which are not necessarily non-degenerate.

\begin{proof}[Proof of Theorem \ref{theorem:spectral_invariants_properties_basic}]
Lipschitz continuity is immediate from Lipschitz continuity of the restriction of $c_{d,k}$ to all pairs of non-degenerate Hamiltonians. Shift, non-negativity, monotonicity, the packing property and the closed curve upper bound are immediate from continuity and Lemmas \ref{lem:shift_spectral_number}, \ref{lem:non_negativity}, \ref{lem:monotonicity}, \ref{lem:packing} and \ref{lem:closed_curve_upper_bound}.\\
We check property (2). By continuity, it suffices to consider the case of non-degenerate Hamiltonians $H_\pm$. Fix $J\in\MJ(H_+,H_-)$ and a set $P\subset M_\phi$ of cardinality $k$. Choose an arbitrary set $P'\subset M_\phi$ containing $P$ and of cardinality $k'$. Let $\epsilon>0$ be arbitrary. There exists a $J$-curve $u\in\MM^J(H_+,H_-,d,P')$ such that $\ME_{H_+,H_-}(u)< c_{d,k'}(H_+,H_-)+\epsilon$. Since $P\subset P'$, the curve $u$ is also contained in $\MM^J(H_+,H_-,d,P)$. This shows that $c_{d,k}(H_+,H_-)\leq c_{d,k'}(H_+,H_-)+\epsilon$. Since $\epsilon>0$ was arbitrary, the desired inequality follows.\\
We show sublinearity. Again we may assume that $H_\pm$ are non-degenerate. Fix $J\in\MJ(H_+,H_-)$ and $P\subset M_\phi$ of cardinality $k$. Moreover, let $\epsilon>0$ be arbitrary. Let $P=P_1\cup P_2$ be an arbitrary partition into two subsets of cardinalities $k_1$ and $k_2$, respectively. There exists curves $u_j\in\MM^J(H_+,H_-,d_j,P_j)$ for $j\in\{1,2\}$ such that $\ME_{H_+,H_-}(u_j)\leq c_{d_j,k_j}(H_+,H_-)+\epsilon/2$. Let $u$ be the union of $u_1$ and $u_2$. Then $u\in\MM^J(H_+,H_-,d,P)$ and
\begin{equation}
\ME_{H_+,H_-}(u)= \ME_{H_+,H_-}(u_1) + \ME_{H_+,H_-}(u_1) \leq c_{d_1,k_1}(H_+,H_-) + c_{d_2,k_2}(H_+,H_-) + \epsilon.
\end{equation}
Sublineariy follows since $\epsilon>0$ was arbitrary.\\
It remains to prove spectrality. If $\phi$ is rational, then the relative action spectrum $\on{Spec}_d(H_+,H_-)$ is a closed subset of $\BR$. This immediately implies that $c_{d,k}(H_+,H_-)\in\on{Spec}_d(H_+,H_-)$ for all non-degenerate Hamiltonians $H_\pm$. Now suppose that $H_\pm$ are not necessarily non-degenerate. Pick sequences of non-degenerate Hamiltonians $H_\pm^j$ converging to $H_\pm$ in the $C^\infty$-topology. By the above, we may pick orbit sets $\alpha_\pm^j$ in $\on{gr}(H_\pm^j)$ of degree $d$ and relative homology classes $Z^j\in H_2(M_\phi,\alpha_+^j,\alpha_-^j)$ such that $\langle Z^j,\Omega_\phi\rangle = c_{d,k}(H_+^j,H_-^j)$. Since the orbit sets $\alpha_\pm^j$ have bounded degree, we may use the Azela-Ascoli theorem to pass to a subsequence converging to orbit sets $\alpha_\pm$ of $Y_{\phi_{H_\pm}}$. This shows that the distance of the sequence of numbers $c_{d,k}(H_+^j,H_-^j)$ to the relative action spectrum $\on{Spec}_d(H_+,H_-)$ converges to $0$. It follows from the fact that $\on{Spec}_d(H_+,H_-)$ is closed that the limit $c_{d,k}(H_+,H_-)$ of the sequence $c_{d,k}(H_+^j,H_-^j)$ is contained in the relative action spectrum.
\end{proof}

\begin{proof}[Proof of Theorem \ref{theorem:spectral_invariants_properties_weyl}]
The shift property is immediate from the corresponding property of the spectral numbers $c_{d,k}$. Let $H_+\geq H_-$. Then it follows from the definition that $c_d(H_+,H_-)$ is bounded below by $c_{d,0}(H_+,H_-)$. The latter is non-negative by the non-negativity property of the spectral numbers $c_{d,k}$. The vanishing property is a consequence of the non-negativity property and monotonicity. Indeed, monotonicity implies that $c_d(H,H)\geq 2c_d(H,H)$. Using $c_d(H,H)\geq 0$, we deduce that $c_d(H,H)$ vanishes. Next, we prove monotonicity. For all non-negative integers $k_\pm$ we have
\begin{eqnarray*}
c_d(H_+,H_-) \geq c_{d,(k_++k_-)(d-g+1)}(H_+,H_-) -(k_++k_-)A \\
\geq \left(c_{d,k_+(d-g+1)}(H_+,H_0)-k_+A\right)+\left(c_{d,k_-(d-g+1)}(H_0,H_-)-k_-A\right).
\end{eqnarray*}
Here the first inequality follows from the definition of $c_d$ and the second inequality is a consequence of the monotonicity of the spectral numbers $c_{d,k}$. Taking the supremum over all pairs of non-negative integers $k_\pm$ yields monotonicity for $c_d$. Lipschitz continuity can be deduced from the shift property, non-negativity and monotonicity, analogous to the proof of Corollary \ref{cor:lipschitz_continuity}. Spectrality follows from the spectrality of $c_{d,k}$ together with the fact that if $\phi$ is rational, the relative action spectrum $\on{Spec}(H_+,H_-,d)$ is closed.\\
If remains to prove the Weyl law. We begin by observing that it suffices to show the inequality
\begin{equation}
\label{eq:weyl_law_proof_a}
\liminf_{i\rightarrow\infty} \frac{c_{d_i}(H_+,H_-)}{d_i} \geq A^{-1} \int_{Y_\phi} (H_+-H_-) dt\wedge\omega_\phi.
\end{equation}
Indeed, interchanging the roles of $H_\pm$ yields the inequality
\begin{equation}
\label{eq:weyl_law_proof_b}
\limsup_{i\rightarrow\infty} -\frac{c_{d_i}(H_-,H_+)}{d_i} \leq A^{-1} \int_{Y_\phi} (H_+-H_-) dt\wedge\omega_\phi.
\end{equation}
Using monotonicity and the vanishing property we deduce
\begin{equation}
0 = c_{d_i}(H_+,H_+) \geq c_{d_i}(H_+,H_-) + c_{d_i}(H_-,H_+)
\end{equation}
which in particular implies
\begin{equation}
\limsup_{i\rightarrow\infty}\frac{c_{d_i}(H_+,H_-)}{d_i} \leq \limsup_{i\rightarrow\infty} -\frac{c_{d_i}(H_-,H_+)}{d_i}.
\end{equation}
Together with \eqref{eq:weyl_law_proof_a} and \eqref{eq:weyl_law_proof_b} this implies the Weyl law.\\
Let us now turn to the proof of \eqref{eq:weyl_law_proof_a}. It follows from the shift property that both sides of \eqref{eq:weyl_law_proof_a} change by the same amount if we replace $H_+$ by $C+H_+$ for some constant $C\in\BR$. Thus we may assume w.l.o.g. that $H_+>H_-$. Let us abbreviate the volume of the symplectic cobordism $M_{H_+,H_-}$ by $V$. Then the right hand side of \eqref{eq:weyl_law_proof_a} is simply given by $A^{-1}V$. Let $\epsilon>0$ be arbitrary. Choose a subset $B\subset\on{int}(M_{H_+,H_-})$ which consists of finitely many disjoint symplectically embedded balls and whose volume $V'$ satisfies $V'\geq V-\epsilon$. For $i>0$, we set
\begin{equation}
k_i\coloneqq \left\lfloor \frac{Vd_i}{A^2}\right\rfloor \qquad\text{and}\qquad \ell_i \coloneqq k_i(d_i-g+1).
\end{equation}
It follows from the definition of the spectral numbers $c_d$ and the packing property of the spectral numbers $c_{d,k}$ that
\begin{equation}
c_{d_i}(H_+,H_-) \geq c_{\ell_i}^H(B) - k_i A.
\end{equation}
Using the Weyl law for Hutchings' capacities $c_k^H$ (see \cite{Hut22}), we compute
\begin{equation}
\lim_{i\rightarrow\infty} \frac{c_{\ell_i}^H(B) - k_i A}{d_i} = \lim_{i\rightarrow \infty} \frac{2\cdot V'^{1/2}\cdot \ell_i^{1/2} - k_i A + O(\ell_i^{1/4})}{d_i} = \frac{2(VV')^{1/2}-V}{A}.
\end{equation}
We deduce that
\begin{equation}
\liminf_{i\rightarrow\infty} \frac{c_{d_i}(H_+,H_-)}{d_i} \geq \frac{2(V(V-\epsilon))^{1/2}-V}{A}
\end{equation}
for all $\epsilon>0$. This implies \eqref{eq:weyl_law_proof_a}.
\end{proof}

\begin{theorem}
\label{theorem:finiteness_for_hamiltonian_diffeomorphisms}
Suppose that $\phi$ is a Hamiltonian diffeomorphism of $(\Sigma,\omega)$. Then the invariants $c_{d,k}^\phi$ and $c_d^\phi$ are finite for all positive integers $d\geq g$ and non-negative integers $k$. Moreover, we have the following estimate:
\begin{equation}
\label{eq:finiteness_for_hamiltonian_diffeomorphisms}
c_{d,1}^\phi(0,0)\leq \frac{A}{d-g+1}.
\end{equation}
\end{theorem}

\begin{proof}
Our strategy is to apply the closed curve upper bound in combination with Theorem \ref{theorem:non_vanishing_taubes_gromov}. Set $X\coloneqq \Sigma\times S^2$ and equip it with the symplectic form $\Omega\coloneqq \omega\oplus\sigma$ where $\sigma$ is an area form on $S^2$ of total area $B$. Our first step is to show that there exists an embedding $\iota:Y_\phi\rightarrow X$ such that $\iota^*\Omega = \omega_\phi$ if $B$ is sufficiently large. We observe that the symplectization of the mapping torus of the identity $\on{id}_\Sigma$ is simply given by
\begin{equation}
M_{\on{id}_\Sigma} = \BR\times \BR/\BZ \times \Sigma\qquad \Omega_{\on{id}_\Sigma} = ds\wedge dt + \omega.
\end{equation}
Since $\phi$ is a Hamiltonian diffeomorphism, we may choose a Hamiltonian $H:\BR/\BZ\times \Sigma\rightarrow\BR$ such that $\phi=\phi_H^1$. The mapping torus $Y_\phi$ can be identified with the graph $\on{gr}(H)\subset M_{\on{id}_\Sigma}$. Choose $B>0$ sufficiently large such that $\max |H|< B/2$. Then $Y_\phi$ is contained in the open subset
\begin{equation}
(-B/2,B/2)\times \BR/\BZ\times \Sigma\subset M_{\on{id}_\Sigma}.
\end{equation}
This set is symplectomorphic to the open subset $\Sigma\times (S^2\setminus \{p_\pm\})\subset X$ where $p_\pm$ denote the north and south pole of the sphere $S^2$. This shows that $(Y_\phi,\omega_\phi)$ embeds into $(X,\Omega)$.\\
Let $c$ be an arbitrary non-negative integer and define the homology class $Z\in H_2(X)$ by \eqref{eq:non_vanishing_taubes_gromov_definition_homology_class}. By assumption $d\geq g$ and hence the condition $d\geq (g-1)\frac{c}{c+1}$ in Theorem \ref{theorem:non_vanishing_taubes_gromov} holds automatically true. Thus Theorem \ref{theorem:non_vanishing_taubes_gromov} implies that $\on{Gr}(X,\Omega;Z)\neq 0$. We have
\begin{equation}
I(Z) = \langle Z,c_1(X)\rangle + Z\cdot Z = 2(1-g)c + 2d + 2cd = 2((d-g+1)c + d)
\end{equation}
and
\begin{equation}
\langle Z,\Omega \rangle = Ac + Bd.
\end{equation}
Clearly, the homological intersection number of $Z$ with a fibre of $Y_\phi\rightarrow\BR/\BZ$ is equal to $d$. We apply the closed curve upper bound and obtain
\begin{equation}
\label{eq:finiteness_for_hamiltonian_diffeomorphisms_proof_a}
c_{d,(d-g+1)c + d}(0,0) \leq Ac + Bd.
\end{equation}
Since $c_{d,k}$ is increasing in $k$ and $c$ can be chosen arbitrarily large, this implies that $c_{d,k}(0,0)$ is finite for all non-negative $k$. By Corollary \ref{cor:finite_or_infinite}, this implies finiteness of $c_{d,k}$ for all pairs of Hamiltonians. Using monotonicity of the spectral invariants $c_{d,k}$, we can estimate
\begin{equation}
c_{d,1}(0,0) \leq \frac{c_{d,(d-g+1)c + d}(0,0)}{(d-g+1)c + d} \leq \frac{Ac + Bd}{(d-g+1)c + d} \longrightarrow \frac{A}{d-g+1}\quad (c\rightarrow\infty).
\end{equation}
This yields the desired inequality \eqref{eq:finiteness_for_hamiltonian_diffeomorphisms}. Finiteness of $c_d$ is immediate from \eqref{eq:finiteness_for_hamiltonian_diffeomorphisms_proof_a} and the fact that $c_{d,k}$ is increasing in $k$.
\end{proof}

\section{Applications}
\label{sec:applications}

\subsection{Quantitative closing lemmas}
\label{subsec:quantitative_closing_lemma}

In this section, we use our elementary spectral invariants to prove Theorem \ref{theorem:quantitative_closing_lemma_hamiltonian_diffeomorphisms}, i.e.\ the quantitative closing lemma for Hamiltonian diffeomorphisms of surfaces. Our proof mirrors the arguments in \cite{EH21} based on PFH spectral invariants.

\begin{definition}
\label{def:spectral_gap}
Let $\phi$ be an area preserving diffeomorphism of $(\Sigma,\omega)$ and let $d\geq 1$ be a positive integer. We define the {\it spectral gap} of $\phi$ in degree $d$ to be the number
\begin{equation}
\on{gap}_d(\phi) \coloneqq c_{d,1}^\phi(0,0) \in [0,\infty].
\end{equation}
\end{definition}

The following result is the analogue of \cite[Proposition 6.2]{EH21}.

\begin{prop}
\label{prop:from_spectral_gap_to_closing_lemma}
Suppose that $\phi$ is a rational area-preserving diffeomorphism of $(\Sigma,\omega)$. Let $\MU\subset\Sigma$ be a nonempty open subset and let $H$ be a $(\MU,a,l)$-admissible Hamiltonian. Assume that $\on{gap}_d(\phi)< a$. Then there exists $\tau\in [0,l^{-1}\on{gap}_d(\phi)]$ such that $\phi_{\tau H}$ has a periodic orbit of period at most $d$ intersecting the open set $\MU$.
\end{prop}

\begin{proof}
We assume by contradiction that $\phi_{\tau H}$ does not have a periodic orbit of period at most $d$ intersecting $\MU$ for any $\tau\in [0,l^{-1}\on{gap}_d(\phi)]$. We choose an open subset $\MV\subset\Sigma$ such that $\ol{\MV}\subset\MU$ and $H$ is also $(\MV,a,l)$-admissible. We claim that there exists $\epsilon>0$ such that there are no periodic orbits of period at most $d$ intersecting $\MV$ for any $\tau\in [0,l^{-1}\on{gap}_d(\phi)+\epsilon]$. This follows from a simple compactness argument. As a consequence, the relative action spectrum $\on{Spec}_d(\tau H,0)$ remains constant as $\tau$ ranges over $[0,l^{-1}\on{gap}_d(\phi)+\epsilon]$. It follows from continuity and spectrality of the invariant $c_{d,1}$ and the fact that the relative action spectrum has measure zero that $c_{d,1}(\tau H,0)$ must be constant in $\tau$ as well. We conclude that
\begin{equation}
\label{eq:from_spectral_gap_to_closing_lemma_proof_a}
\on{gap}_d(\phi) = c_{d,1}(0,0) = c_{d,1}(\delta H,0)
\end{equation}
where we abbreviate $\delta\coloneqq l^{-1}\on{gap}_d(\phi)+\epsilon$. Set $b\coloneqq \min(\delta l,a)$. Then the ball $B^4(b)$ symplectically embeds into $M_{\delta H,0}$. Equation \eqref{eq:from_spectral_gap_to_closing_lemma_proof_a} and the packing property imply that
\begin{equation}
\on{gap}_d(\phi) = c_{d,1}(\delta H,0) \geq c_1^H(B^4(b)) = b = \min(\on{gap}_d(\phi) + \epsilon l,a) > \on{gap}_d(\phi).
\end{equation}
This is a contradiction.
\end{proof}

Proposition \ref{prop:from_spectral_gap_to_closing_lemma} reduces the proof of quantitative closing lemmas to establishing the existence of arbitrarily small spectral gaps. If $\phi$ is a Hamiltonian diffeomorphism of $(\Sigma,\omega)$, we are able to do so without appealing to Floer theory (cf. \cite[Lemma 7.1]{EH21}).

\begin{lem}
\label{lem:small_spectral_gaps_for_hamiltonian_diffeomorphisms}
Let $d\geq g$ be a positive integer. Then
\begin{equation}
\label{eq:small_spectral_gaps_for_hamiltonian_diffeomorphisms}
\on{gap}_d(\phi) \leq \frac{A}{d-g+1}
\end{equation}
for every Hamiltonian diffeomorphism $\phi$ of $(\Sigma,\omega)$.
\end{lem}

\begin{proof}
This is immediate from Theorem \ref{theorem:finiteness_for_hamiltonian_diffeomorphisms} and the definition of $\on{gap}_d(\phi)$.
\end{proof}

\begin{remark}
\label{remark:small_spectral_gap_rational}
It follows from monotonicity of the invariants $c_{d,k}$ and \eqref{eq:upper_bound_elementary_pfh_for_u_cyclic} that \eqref{eq:small_spectral_gaps_for_hamiltonian_diffeomorphisms} continues to hold for all $\phi$ and $d$ such that there exists a degree $d$ reference cycle $\gamma$ with the property that $HP(\phi,\gamma)$ contains $U$-cyclic classes. By \cite{CPPZ21} this is the case for all rational $\phi$ and all $d > \max\{2g-2,0\}$ such that there exists a monotone reference cycle in degree $d$. Note that for every rational $\phi$ the set of such $d$ is an infinite arithmetic progression.
\end{remark}

\begin{proof}[Proof of Theorem \ref{theorem:quantitative_closing_lemma_hamiltonian_diffeomorphisms}]
Set $d_0\coloneqq g + \lfloor Al^{-1}\delta^{-1}\rfloor$. We have
\begin{equation}
d_0-g+1 > A\l^{-1}\delta^{-1} \geq Aa^{-1}.
\end{equation}
In particular, we see that
\begin{equation}
\label{eq:quantitative_closing_lemma_hamiltonian_diffeomorphisms_proof}
d_0 > g \qquad\text{and}\qquad \frac{A}{d_0-g+1}<a\qquad\text{and}\qquad l^{-1}\frac{A}{d_0-g+1} < \delta.
\end{equation}
By Lemma \ref{lem:small_spectral_gaps_for_hamiltonian_diffeomorphisms}, we have $\on{gap}_{d_0}(\phi) \leq \frac{A}{d_0-g+1}$. In combination with \eqref{eq:quantitative_closing_lemma_hamiltonian_diffeomorphisms_proof}, this shows that $\on{gap}_{d_0}(\phi)< a$ and $l^{-1}\on{gap}_{d_0}(\phi)<\delta$. Hence Proposition \ref{prop:from_spectral_gap_to_closing_lemma} implies that there exists $\tau\in [0,\delta]$ such that $\phi_{\tau H}$ has a periodic orbit of period at most $d_0$ intersection $\MU$.
\end{proof}

\subsection{Simplicity conjecture and extension of Calabi}
\label{subsec:simplicity_conjecture_and_calabi}

In this section, we prove Theorems \ref{theorem:simplicity_conjecture} (simplicity conjecture) and \ref{theorem:calabi_extension} (extension of Calabi). In order to do so, it will be convenient to define a third family of spectral invariants
\begin{equation}
\gamma_d:\on{Ham}(\Sigma,\omega)\longrightarrow\BR
\end{equation}
taking Hamiltonian diffeomorphisms of $(\Sigma,\omega)$ as an input (cf. \cite[section 3.2]{CHS21}). The simplicity conjecture and the extension of the Calabi homomorphism follow from the basic properties of the spectral invariants $\gamma_d$ (see Theorems \ref{theorem:spectral_invariants_properties_diffeomorphisms_nonempty_boundary} and \ref{theorem:spectral_invariants_properties_diffeomorphisms_closed} below) in essentially the same way these results are deduced from the properties of PFH and Heegaard Floer spectral invariants in \cite{CHS20}, \cite{CHS21} and \cite{CHMSS21}. For the sake of completeness, we carry out versions of these arguments based on the spectral invariants $\gamma_d$. Similarly to subsection \ref{subsec:simplicity_conjecture_and_extension_of_calabi}, the surface $\Sigma$ is allowed to have boundary in the present subsection.

\begin{remark}
Strictly speaking, in the case $\Sigma=S^2$ the invariant $\gamma_d$ is defined on the universal cover $\wt{\on{Ham}}(S^2,\omega)$ and might not descend to $\on{Ham}(S^2,\omega)$. Note that if $\Sigma\neq S^2$, then $\on{Ham}(\Sigma,\omega)$ is simply connected and hence agrees with its universal cover. In the case $\Sigma=S^2$, the group $\on{Ham}(S^2,\omega)$ is homotopy equivalent to $\on{SO}(3)$. Thus $\wt{\on{Ham}}(S^2,\omega)$ is a double cover of $\on{Ham}(S^2,\omega)$. In order to simplify the discussion, we will assume in the following that $\Sigma\neq S^2$. It is not hard to extend our proof of the simplicity conjecture (Theorem \ref{theorem:simplicity_conjecture}) to the case $\Sigma=S^2$. Note that the extension of Calabi (Theorem \ref{theorem:calabi_extension}) only pertains to surfaces with non-empty boundary anyway.
\end{remark}

Let us first define $\gamma_d$ for closed surfaces $\Sigma$. Given an element $\phi\in\on{Ham}(\Sigma,\omega)$, we choose a Hamiltonian $H\in C^\infty(\BR/\BZ\times\Sigma)=C^\infty(Y_{\on{id}_\Sigma})$ such that $\phi=\phi_H^1$. We normalize $H$ such that
\begin{equation}
\label{eq:normalization_hamiltonian}
\int_{\BR/\BR\times\Sigma}H dt\wedge\omega = 0
\end{equation}
and define
\begin{equation}
\gamma_d(\phi) \coloneqq c_d^{\on{id}_\Sigma}(H,0)
\end{equation}
for every positive integer $d\geq g$. We will see in Theorem \ref{theorem:spectral_invariants_properties_diffeomorphisms_closed} below that this definition does not depend on the choice of $H$. Moreover, it will be useful to define invariants $\eta_d$ (cf.\ \cite[section 3.2.3]{CHS21}) by
\begin{equation}
\label{eq:spectral_invariant_difference_definition}
\eta_d : \on{Ham}(\Sigma,\omega)\rightarrow \BR\qquad \eta_d\coloneqq \frac{\gamma_d}{d}-\gamma_1.
\end{equation}
If $\Sigma$ has non-empty boundary, we can form a closed surface $\hat{\Sigma}$ by collapsing all boundary components to points. We may equip $\hat{\Sigma}$ with a smooth structure and an area form $\hat{\omega}$ such that the natural inclusion of the interior of $\Sigma$ into $\hat{\Sigma}$ is a smooth area-preserving embedding. Given $\phi\in\on{Ham}(\Sigma,\omega)$, we choose a Hamiltonian $H\in C^\infty(\BR/\BZ\times\Sigma)$ which is compactly supported in the interior of $\Sigma$ and generates $\phi$. Since $H$ is supported away from the boundary, it smoothly extends to a Hamiltonian $\hat{H}$ on the closed surface $\hat{\Sigma}$. We define
\begin{equation}
\gamma_d(\phi)\coloneqq c_d^{\on{id}_{\hat{\Sigma}}}(\hat{H},0).
\end{equation}
As for the closed case, we will see below that this does not depend on the choice of $H$.

\begin{theorem}
\label{theorem:spectral_invariants_properties_diffeomorphisms_nonempty_boundary}
Suppose that $\Sigma$ has non-empty boundary. Then the spectral invariants $\gamma_d$ are well-defined, i.e.\ independent of choices, and satisfy the following properties:
\begin{enumerate}
\item {\bf (Hofer Lipschitz)} The invariant $\gamma_d$ is Lipschitz continuous with respect to the Hofer metric with Lipschitz constant $d$.
\item {\bf (Monotonicity)} Let $\phi,\psi\in\on{Ham}(\Sigma)$ and assume that $\phi\leq \psi$, i.e.\ there exist Hamiltonians $H\leq G$ compactly supported in the interior of $\Sigma$ such that $\phi=\phi_H^1$ and $\psi=\phi_G^1$. Then $\gamma_d(\phi)\leq \gamma_d(\psi)$.
\item {\bf (Normalization)} We have $\gamma_d(\on{id}_\Sigma)=0$.
\item {\bf (Asymptotics)} We have
\begin{equation*}
\lim_{d\rightarrow\infty}\frac{\gamma_d(\phi)}{d}= A^{-1}\on{Cal}(\phi)
\end{equation*}
for every $\phi\in\on{Ham}(\Sigma,\omega)$.
\item {\bf (Support control)} Let $\phi\in\on{Ham}(\Sigma,\omega)$. Let $U\subset\Sigma$ be a non-empty open subset and assume that the sets $\phi^j(U)$ for $0\leq j\leq d$ are pairwise disjoint. Let $\psi\in\on{Ham}(U,\omega)$ be a compactly supported Hamiltonian diffeomorphism of $U$. Then
\begin{equation*}
\gamma_d(\phi\circ\psi) = \gamma_d(\phi).
\end{equation*}
\item {\bf ($C^0$ continuity)} $\gamma_d$ admits a unique $C^0$-continuous extension to $\ol{\on{Ham}}(\Sigma,\omega)$.
\end{enumerate}
\end{theorem}

\begin{theorem}
\label{theorem:spectral_invariants_properties_diffeomorphisms_closed}
Suppose that $\Sigma$ is closed. Then the spectral invariants $\gamma_d$ are well-defined, i.e.\ independent of choices, and satisfy the following properties:
\begin{enumerate}
\item {\bf (Hofer Lipschitz)} The invariant $\gamma_d$ is Lipschitz continuous with respect to the Hofer metric with Lipschitz constant $d$.
\item {\bf (Normalization)} We have $\gamma_d(\on{id}_\Sigma)=0$.
\item {\bf (Asymptotics)} We have
\begin{equation*}
\lim_{d\rightarrow\infty}\frac{\gamma_d(\phi)}{d}=0
\end{equation*}
for every $\phi\in\on{Ham}(\Sigma,\omega)$.
\item {\bf (Support control)} Let $\phi\in\on{Ham}(\Sigma,\omega)$. Let $U\subset\Sigma$ be a non-empty open subset and assume that the sets $\phi^j(U)$ for $0\leq j\leq d$ are pairwise disjoint. Let $\psi\in\on{Ham}(U,\omega)$ be a compactly supported Hamiltonian diffeomorphism of $U$. Then
\begin{equation*}
\gamma_d(\phi\circ\psi) = \gamma_d(\phi) - dA^{-1}\on{Cal}(\psi).
\end{equation*}
\item {\bf ($C^0$ continuity of differences)} The spectral invariant $\eta_d$ defined in \eqref{eq:spectral_invariant_difference_definition} admits a unique $C^0$-continuous extension to $\ol{\on{Ham}}(\Sigma,\omega)$.
\end{enumerate}
\end{theorem}

\begin{proof}[Proof of Theorem \ref{theorem:spectral_invariants_properties_diffeomorphisms_nonempty_boundary}]
We begin by showing that $\gamma_d$ is well-defined. By Theorem \ref{theorem:finiteness_for_hamiltonian_diffeomorphisms}, the spectral numbers $c_d^{\on{id}_{\hat{\Sigma}}}$ are finite. Hence $\gamma_d$ indeed takes values in $\BR$. Let $\phi\in\on{Ham}(\Sigma,\omega)$ and let $H^0$ and $H^1$ be two Hamiltonians compactly supported in the interior of $\Sigma$ generating $\phi$. Since $\on{Ham}(\Sigma,\omega)$ is simply connected, there exists a homotopy of compactly supported Hamiltonians $(H^\lambda)_{\lambda\in [0,1]}$ such that $\phi_{H^\lambda}^1=\phi$ for all $\lambda$. Let $\hat{H}^\lambda$ denote the induced Hamiltonians on $\hat{\Sigma}$. We have
\begin{equation}
\int_{\BR/\BZ\times\hat{\Sigma}} \hat{H}^\lambda dt\wedge\omega = \on{Cal}(\phi)
\end{equation}
for all $\lambda$. Thus Lemma \ref{lem:relative_action_spectrum_homotopy_invariance} implies that the relative action spectrum $\on{Spec}_d(\hat{H}^\lambda,0)$ is independent of $\lambda$. The fact the the relative action spectrum has measure zero and continuity and spectrality of the map $\lambda\mapsto c_d^{\on{id}_{\hat{\Sigma}}}(\hat{H}^\lambda,0)$ imply that it must be constant.  Hence $c_d^{\on{id}_{\hat{\Sigma}}}(\hat{H}^0,0)=c_d^{\on{id}_{\hat{\Sigma}}}(\hat{H}^1,0)$, showing that $\gamma_d$ is well-defined.\\
Hofer Lipschitz continuity of the spectral invariants $\gamma_d$ follows from Lipschitz continuity of the spectral invariants $c_d$. Monotonicity of $\gamma_d$ follows from monotonicity and non-negativity of $c_d$. The normalization property of $\gamma_d$ is a consequence of the vanishing property of $c_d$ and the asymptotic formula for $\gamma_d$ is immediate from the Weyl law for $c_d$. Next, we verify the support control property. Choose a Hamiltonian $H:\BR/\BZ\times\Sigma\rightarrow\BR$ generating $\phi$ and with support contained in $(1/2,1)\times\on{int}(\Sigma)$. Moreover, choose a Hamiltonian $G:\BR/\BZ\times\Sigma\rightarrow\BR$ generating $\psi$ with support contained in $(0,1/2)\times U$. Then the composition $\phi\circ\psi$ is generated by $H+G$. For $\lambda\in [0,1]$ we define $K^\lambda\coloneqq H+\lambda G$. Since the sets $\phi^j(U)$ for $0\leq j\leq d$ are disjoint, the set of $1$-cycles in the symplectization $M_{\on{id}_{\hat{\Sigma}}}$ which are represented by degree $d$ orbit sets of the graph $\on{gr}(\hat{K}^\lambda)$ agrees with the set of $1$-cycles represented by degree $d$ orbits sets of $\on{gr}(\hat{H})$ for all $\lambda$. This implies that the relative action spectrum $\on{Spec}_d(\hat{K}^\lambda,0)$ is independent of $\lambda$. Thus it follows from continuity and spectrality of $c_d$ that
\begin{equation}
\gamma_d(\phi\circ\psi) = c_d^{\on{id}_{\hat{\Sigma}}}(\hat{K}^1,0) = c_d^{\on{id}_{\hat{\Sigma}}}(\hat{K}^0,0) = \gamma_d(\phi).
\end{equation}
The proof of $C^0$ continuity is based on the support control property and the following lemma:

\begin{lem}
\label{lem:fragmentation}
Let $B\subset\Sigma$ be an open topological disk. For all $\epsilon>0$, there exists $\delta>0$, such that for all $\phi\in\on{Ham}(\Sigma,\omega)$ with $d_{C^0}(\phi,\on{id}_\Sigma)<\delta$, there is $\psi\in\on{Ham}(\Sigma,\omega)$ supported in $B$ such that $d_H(\phi,\psi) < \epsilon$.
\end{lem}

\begin{proof}
This lemma is proved in \cite[Lemma 3.11]{CHS21} in the case $\Sigma=S^2$. As remarked in \cite{CHMSS21}, the proof carries over to arbitrary $\Sigma$.
\end{proof}

We use these ingredients to show:
\begin{itemize}
\item[($\bullet$)] {\it For every Hamiltonian homoemorphism $\chi \in\ol{\on{Ham}}(\Sigma,\omega)$ and every $\epsilon>0$, there exists $\delta>0$ such that for all $\phi,\psi\in\on{Ham}(\Sigma,\omega)$ we have
\begin{equation}
d_{C^0}(\phi,\chi)<\delta \quad\text{and}\quad d_{C^0}(\psi,\chi)<\delta \qquad\Longrightarrow\qquad |\gamma_d(\phi)-\gamma_d(\psi)| < \epsilon.
\end{equation}}
\end{itemize}
Clearly, this statement implies that $\gamma_d$ admits a $C^0$-continuous extension to $\ol{\on{Ham}}(\Sigma,\omega)$.\\
We claim that it suffices to prove $(\bullet)$ under the following additional assumption:
\begin{itemize}
\item[($\ast$)] There exists a point $x\in\Sigma$ such that the points $\chi^j(x)$ for $0\leq j\leq d$ are pairwise distinct.
\end{itemize}
Indeed, let us assume that ($\bullet$) holds under the assumption ($\ast$). Let $\chi\in\ol{\on{Ham}}(\Sigma,\omega)$ be arbitrary and let $\epsilon>0$. Choose $\rho\in\on{Ham}(\Sigma,\omega)$ such that
\begin{equation}
\label{eq:spectral_invariants_properties_diffeomorphisms_nonempty_boundary_proof_a}
d_H(\rho,\on{id}_\Sigma) < \frac{\epsilon}{3d}
\end{equation}
and ($\ast$) holds for $\chi\circ\rho$. Thus ($\bullet$) holds for $\chi\circ\rho$ by assumption and we may pick $\delta>0$ such that for all $\phi,\psi\in\on{Ham}(\Sigma,\omega)$ satisfying $d_{C^0}(\phi,\chi\circ\rho)<\delta$ and $d_{C^0}(\psi,\chi\circ\rho)<\delta$, we have $|\gamma_d(\phi)-\gamma_d(\psi)|<\frac{\epsilon}{3}$. Now fix $\phi,\psi\in\on{Ham}(\Sigma,\omega)$ such that $d_{C^0}(\phi,\chi)<\delta$ and $d_{C^0}(\psi,\chi)<\delta$. This implies that $d_{C^0}(\phi\circ\rho,\chi\circ\rho)<\delta$ and $d_{C^0}(\psi\circ\rho,\chi\circ\rho)<\delta$ and therefore
\begin{equation}
\label{eq:spectral_invariants_properties_diffeomorphisms_nonempty_boundary_proof_b}
|\gamma_d(\phi\circ\rho)-\gamma_d(\psi\circ\rho)|<\frac{\epsilon}{3}.
\end{equation}
By the triangle inequality we have
\begin{equation}
\label{eq:spectral_invariants_properties_diffeomorphisms_nonempty_boundary_proof_c}
|\gamma_d(\phi)-\gamma_d(\psi)| \leq |\gamma_d(\phi)-\gamma_d(\phi\circ\rho)| + |\gamma_d(\phi\circ\rho)-\gamma_d(\psi\circ\rho)| +|\gamma_d(\psi)-\gamma_d(\psi\circ\rho)|.
\end{equation}
The inequality \eqref{eq:spectral_invariants_properties_diffeomorphisms_nonempty_boundary_proof_a} and Hofer Lipschitz continuity of $\gamma_d$ with Lipschitz constant $d$ imply that the first and the third term on the right hand side of \eqref{eq:spectral_invariants_properties_diffeomorphisms_nonempty_boundary_proof_c} are bounded above by $\frac{\epsilon}{3}$. By \eqref{eq:spectral_invariants_properties_diffeomorphisms_nonempty_boundary_proof_b} the second term is bounded above by $\frac{\epsilon}{3}$ as well. This shows that ($\bullet$) holds for $\chi$.\\
It remains to prove ($\bullet$) under assumption ($\ast$). Fix $\epsilon>0$ and $\chi\in\ol{\on{Ham}}(\Sigma,\omega)$ satisfying ($\ast$). By ($\ast$), we may choose an open topological disk $B\subset\Sigma$ such that the closures of $\chi^j(B)$ for $0\leq j\leq d$ are pairwise disjoint. It follows from Lemma \ref{lem:fragmentation} that there exists $\delta>0$ with the following two properties:
\begin{enumerate}
\item For all $\phi\in\on{Ham}(\Sigma,\omega)$ satisfying $d_{C^0}(\phi,\chi)<\delta$ the sets $\phi^j(B)$ for $0\leq j\leq d$ are pairwise disjoint.
\item For all $\phi,\psi\in\on{Ham}(\Sigma,\omega)$ satisfying $d_{C^0}(\phi,\chi)<\delta$ and $d_{C^0}(\psi,\chi)<\delta$ there exists $\rho\in\on{Ham}(\Sigma,\omega)$ with support contained in $B$ such that
\begin{equation}
\label{eq:spectral_invariants_properties_diffeomorphisms_nonempty_boundary_proof_d}
d_H(\rho,\phi^{-1}\psi) < \frac{\epsilon}{d}.
\end{equation}
\end{enumerate}
Given $\phi$ and $\psi$ satisfying $d_{C^0}(\phi,\chi)<\delta$ and $d_{C^0}(\psi,\chi)<\delta$, pick $\rho$ as in (2) and estimate
\begin{IEEEeqnarray}{rCl}
|\gamma_d(\phi)-\gamma_d(\psi)| & \leq & |\gamma_d(\phi)-\gamma_d(\phi\circ\rho)| + |\gamma_d(\phi\circ\rho)-\gamma_d(\phi\circ\phi^{-1}\circ\psi)|\nonumber\\
& = & |\gamma_d(\phi\circ\rho)-\gamma_d(\phi\circ\phi^{-1}\circ\psi)|\nonumber\\
& \leq & d\cdot d_H(\phi\circ\rho,\phi\circ(\phi^{-1}\circ\psi))\nonumber\\
& = & d\cdot d_H(\rho,\phi^{-1}\circ\psi)\nonumber\\
& < & \epsilon. \nonumber
\end{IEEEeqnarray}
Here the first inequality is simply the triangle inequality. The equality uses (1) and the support control property applied to $\phi$ and $\rho$. The second inequality follows from Hofer Lipschitz continuity of $\gamma_d$. The second equality uses bi-invariance of Hofer's metric and the final inequality follows from \eqref{eq:spectral_invariants_properties_diffeomorphisms_nonempty_boundary_proof_d}. This concludes the proof of ($\bullet$) and hence of the $C^0$ continuity property of the spectral invariant $\gamma_d$.
\end{proof}

\begin{proof}[Proof of Theorem \ref{theorem:spectral_invariants_properties_diffeomorphisms_closed}]
The argument proving that $\gamma_d$ is well-defined in the case of non-empty boundary carries over to the closed case almost verbatim. Hofer Lipschitz continuity of $\gamma_d$ is immediate from the corresponding property of $c_d$. Normalization of $\gamma_d$ follows from the vanishing property of $c_d$. The asymptotic formula is a consequence of the Weyl law for $c_d$ and the normalization \eqref{eq:normalization_hamiltonian} used in the definition of $\gamma_d$. Next, we check the support control property. Choose a Hamiltonian $H:\BR/\BZ\times\Sigma\rightarrow\BR$ which generates $\phi$, has support contained in $(1/2,1)\times\Sigma$ and satisfies the normalization \eqref{eq:normalization_hamiltonian}. Moreover, choose a Hamiltonian $G:\BR/\BZ\times\Sigma\rightarrow\BR$ generating $\psi$ with support contained in $(0,1/2)\times U$. Then the composition $\phi\circ\psi$ is generated by $H+G$. For $\lambda\in [0,1]$ we define $K^\lambda\coloneqq H+\lambda G$. Since the sets $\phi^j(U)$ for $0\leq j\leq d$ are disjoint, the set of $1$-cycles in the symplectization $M_{\on{id}_{\Sigma}}$ which are represented by degree $d$ orbit sets of the graph $\on{gr}(K^\lambda)$ agrees with the set of $1$-cycles represented by degree $d$ orbits sets of $\on{gr}(H)$ for all $\lambda$. This implies that the relative action spectrum $\on{Spec}_d(K^\lambda,0)$ is independent of $\lambda$. Thus it follows from continuity and spectrality of $c_d$ that
\begin{equation}
\label{eq:spectral_invariants_properties_diffeomorphisms_closed_proof_a}
c_d^{\on{id}_{\Sigma}}(H+G,0) = c_d^{\on{id}_{\Sigma}}(K^1,0) = c_d^{\on{id}_{\Sigma}}(K^0,0) = \gamma_d(\phi).
\end{equation}
Since the shift $H+G-A^{-1}\on{Cal}(\psi)$ satisfies normalization \eqref{eq:normalization_hamiltonian}, we have
\begin{IEEEeqnarray}{rCl}
\gamma_d(\phi\circ\psi)  & = & c_d^{\on{id}_\Sigma}(H+G-A^{-1}\on{Cal}(\psi),0) \nonumber\\
& = & c_d^{\on{id}_{\Sigma}}(H+G,0) - dA^{-1}\on{Cal}(\psi) \nonumber\\
& = & \gamma_d(\phi) - dA^{-1}\on{Cal}(\psi). \nonumber
\end{IEEEeqnarray}
Here the first equality uses the definition of $\gamma_d$, the second equality uses the shift property of $c_d$ and the third equality follows from \eqref{eq:spectral_invariants_properties_diffeomorphisms_closed_proof_a}. It remains to check $C^0$ continuity of $\eta_d$. The support control property of $\gamma_d$ implies that the spectral invariants $\eta_d$ satisfy $\eta_d(\phi\circ\psi) = \eta_d(\phi)$. Using this, the argument for $C^0$-continuity in the case of non-empty boundary carries over almost word for word.
\end{proof}

\begin{lem}
\label{lem:relationship_spectral_invariants_closed_nonempty_boundary}
Suppose that $\Sigma$ has non-empty boundary. Let $(\hat{\Sigma},\hat{\omega})$ be the closed surface obtained by collapsing the boundary components of $\Sigma$ as explained in the definition of $\gamma_d$. Assume that $\hat{\Sigma}\neq S^2$. Let $\phi\in\on{Ham}(\Sigma)$. Since $\phi$ is supported away form the boundary of $\Sigma$, it extends to $\hat{\phi}\in\on{Ham}(\hat{\Sigma},\hat{\omega})$. We have
\begin{equation}
\gamma_d(\hat{\phi}) = \gamma_d(\phi) - dA^{-1}\on{Cal}(\phi).
\end{equation}
\end{lem}

\begin{proof}
Let $H$ be a Hamiltonian which is compactly supported in the interior of $\Sigma$ and generates $\phi$. Let $\hat{H}$ denote the Hamiltonian on $\hat{\Sigma}$ induced by $H$. By definition, we have $\gamma_d(\phi) = c_d^{\on{id}_{\hat{\Sigma}}}(\hat{H},0)$. Define the Hamiltonian $G$ on $\hat{\Sigma}$ by
\begin{equation}
G\coloneqq \hat{H}- A^{-1}\on{Cal}(\phi).
\end{equation}
Clearly, $G$ satisfies the normalization \eqref{eq:normalization_hamiltonian} and generates $\hat{\phi}$. Thus
\begin{equation}
\gamma_d(\hat{\phi}) = c_d^{\on{id}_{\hat{\Sigma}}}(G,0) = c_d^{\on{id}_{\hat{\Sigma}}}(\hat{H},0) - dA^{-1}\on{Cal}(\phi) = \gamma_d(\phi)-dA^{-1}\on{Cal}(\phi).
\end{equation}
Here the first equality follows from the definition of $\gamma_d$, the second equality uses the shift property of the spectral invariants $c_d$ and the third equality again uses the definition of $\gamma_d$.
\end{proof}

Now we are in a position to prove the simplicity conjecture and the extension of the Calabi homomorphism.

\begin{proof}[Proof of Theorem \ref{theorem:simplicity_conjecture} (simplicity conjecture)]
Our goal is to show that the group $\on{FHomeo}(\Sigma,\omega)$ of finite energy homeomorphisms is a proper subgroup of $\ol{\on{Ham}}(\Sigma,\omega)$. Following \cite{CHS20}, we construct an explicit Hamiltonian homeomorphism $\phi\in\ol{\on{Ham}}(\Sigma,\omega)$ and use the spectral invariants $\gamma_d$ to show that it is not contained in $\on{FHomeo}(\Sigma)$. For $R>0$, let $D_R\subset\BR^2$ denote the disk of radius $R$ centered at the origin. If $R>0$ is sufficiently small, there exists an area-preserving embedding of $D_R$ into $\Sigma$. We fix such an embedding and identify $D_R$ with its image under this embedding. We choose a smooth, decreasing function $f:(0,R]\rightarrow\BR_{\geq 0}$ vanishing in a neighbourhood of $R$ such that
\begin{equation}
\label{eq:simplicity_conjecture_proof_a}
\int_0^R r^3f(r)dr = +\infty.
\end{equation}
Using polar coordinates $(r,\theta)$ on $D_R$, we define $\phi$ to be
\begin{equation}
\label{eq:simplicity_conjecture_proof_b}
\phi(r,\theta) \coloneqq (r,\theta + f(r))
\end{equation}
on $D_R$. Outside $D_R$, we set $\phi$ to be the identity. In order to check that $\phi$ is indeed a Hamiltonian homeomorphism, we choose an increasing sequence of smooth, decreasing functions $f_j:[0,R]\rightarrow\BR_{\geq 0}$ converging to $f$ pointwise on $(0,R]$ such that each $f_j$ is constant near $0$ and vanishes near $R$. We define Hamiltonian diffeomorphisms $\phi_j\in\on{Ham}(\Sigma,\omega)$ by
\begin{equation}
\phi_j(r,\theta) \coloneqq (r,\theta + f_j(r))
\end{equation}
on $D_R$ and set $\phi_j$ to be the identity outside $D_R$. Clearly, $\phi_j$ converges to $\phi$ with respect to $d_{C^0}$, showing that $\phi\in\ol{\on{Ham}}(\Sigma,\omega)$. It is an immediate consequence of Lemmas \ref{lem:infinite_twist} and \ref{lem:finiteness_on_finite_energy_homeomorphisms} below that $\phi$ is not contained in $\on{FHomeo}(\Sigma,\omega)$.
\end{proof}

\begin{lem}
\label{lem:infinite_twist}
Let $\phi\in\ol{\on{Ham}}(\Sigma,\omega)$ be the Hamiltonian homeomorphism defined in \eqref{eq:simplicity_conjecture_proof_b}. If $\Sigma$ has non-empty boundary, then
\begin{equation}
\label{eq:infinite_twist_nonempty_boundary}
\sup_{d\geq 1} \left| \frac{\gamma_d(\phi)}{d}\right| = +\infty.
\end{equation}
If $\Sigma$ is closed, then
\begin{equation}
\label{eq:infinite_twist_closed}
\sup_{d\geq 1} \left| \eta_d(\phi) \right| = +\infty.
\end{equation}
\end{lem}

\begin{proof}
Let us first treat the case that $\Sigma$ has non-empty boundary. It follows from a direct computation that
\begin{equation}
\int_0^R r^3 f_j(r) dr = \on{Cal}(\phi_j).
\end{equation}
Thus \eqref{eq:simplicity_conjecture_proof_a} implies that
\begin{equation}
\label{eq:infinite_twist_proof_a}
\on{Cal}(\phi_j)\rightarrow +\infty \qquad (j\rightarrow\infty).
\end{equation}
Assume by contradiction that the supremum in \eqref{eq:infinite_twist_nonempty_boundary} is equal to some finite real number $C\geq 0$. It follows from monotonicity and $C^0$-continuity in Theorem \ref{theorem:spectral_invariants_properties_diffeomorphisms_nonempty_boundary} that, for each fixed $d$, the sequence $\frac{\gamma_d(\phi_j)}{d}$ is increasing and converges to $\frac{\gamma_d(\phi)}{d}$ as $j\rightarrow\infty$. Hence $\frac{\gamma_d(\phi_j)}{d}\leq C$ for all $d$ and $j$. By the asymptotic formula in Theorem \ref{theorem:spectral_invariants_properties_diffeomorphisms_nonempty_boundary} we have
\begin{equation}
A^{-1}\on{Cal}(\phi_j) = \lim_{d\rightarrow\infty} \frac{\gamma_d(\phi_j)}{d} \leq C
\end{equation}
for all $j$. This contradicts \eqref{eq:infinite_twist_proof_a}.\\
We use Lemma \ref{lem:relationship_spectral_invariants_closed_nonempty_boundary} to reduce the closed case to the case with non-empty boundary. We pick a point $p\in\Sigma$ not contained in the support of $\phi$. Blowing this point up to a circle yields a surface $\Sigma'$ with non-empty boundary such that $\Sigma=\hat{\Sigma}'$. Let $\phi_j'\in\on{Ham}(\Sigma',\omega')$ and $\phi'\in\ol{\on{Ham}}(\Sigma',\omega')$ be the Hamiltonian diffeomorphisms/homeomorphisms induced by $\phi_j$ and $\phi$, respectively. We compute
\begin{IEEEeqnarray}{rCl}
\eta_d(\phi_j) & = & \frac{\gamma_d(\phi_j)}{d}-\gamma_1(\phi_j) \nonumber\\
& = & \frac{\gamma_d(\phi_j')-dA^{-1}\on{Cal}(\phi_j')}{d}-(\gamma_1(\phi_j')-A^{-1}\on{Cal}(\phi_j')) \nonumber\\
& = & \frac{\gamma_d(\phi_j')}{d}-\gamma_1(\phi_j'). \nonumber
\end{IEEEeqnarray}
Here the first equality simply is the definition of $\eta_d$ and the seceond equality uses Lemma \ref{lem:relationship_spectral_invariants_closed_nonempty_boundary}. Taking the limit $j\rightarrow\infty$ on both sides and using $C^0$-continuity in Theorems \ref{theorem:spectral_invariants_properties_diffeomorphisms_closed} and \ref{theorem:spectral_invariants_properties_diffeomorphisms_nonempty_boundary} yields
\begin{equation}
\label{eq:infinite_twist_proof_b}
\eta_d(\phi) = \frac{\gamma_d(\phi')}{d}-\gamma_1(\phi').
\end{equation}
The desired identity \eqref{eq:infinite_twist_closed} follows from \eqref{eq:infinite_twist_proof_b} and the case with non-empty boundary \eqref{eq:infinite_twist_nonempty_boundary}.
\end{proof}

\begin{lem}
\label{lem:finiteness_on_finite_energy_homeomorphisms}
Let $\psi\in\on{FHomeo}(\Sigma,\omega)$ be a finite energy homeomorphism. If $\Sigma$ has non-empty boundary, then
\begin{equation}
\label{eq:finiteness_on_finite_energy_homeomorphisms_nonempty_boundary}
\sup_{d\geq 1} \left| \frac{\gamma_d(\psi)}{d}\right| < +\infty.
\end{equation}
If $\Sigma$ is closed, then
\begin{equation}
\label{eq:finiteness_on_finite_energy_homeomorphisms_closed}
\sup_{d\geq 1} \left| \eta_d(\psi) \right| < +\infty.
\end{equation}
\end{lem}

\begin{proof}
By definition of $\on{FHomeo}(\Sigma,\omega)$, there exists a sequence $(\psi_j)_{j>0}$ in $\on{Ham}(\Sigma)$ which is bounded with respect to Hofer's metric $d_H$ and converges to $\psi$ with respect to $d_{C^0}$. Hofer Lipschitz continuity of $\gamma_d$ and the normalization property $\gamma_d(\on{id}_\Sigma)=0$ imply that there exists a constant $C>0$ such that
\begin{equation}
\label{eq:finiteness_on_finite_energy_homeomorphisms_proof_a}
|\gamma_d(\psi_j)|\leq Cd
\end{equation}
for all $d$ and $j$. In the case that $\Sigma$ has non-empty boundary, $C^0$-continuity of $\gamma_d$ implies that $\gamma_d(\psi)=\lim_{j\rightarrow\infty} \gamma_d(\psi_j)$. Thus finiteness of \eqref{eq:finiteness_on_finite_energy_homeomorphisms_nonempty_boundary} is an immediate consequence of \eqref{eq:finiteness_on_finite_energy_homeomorphisms_proof_a}. If $\Sigma$ is closed, then \eqref{eq:finiteness_on_finite_energy_homeomorphisms_proof_a} implies that $|\eta_d(\psi_j)|\leq 2C$ for all $d$ and $j$. Finiteness of \eqref{eq:finiteness_on_finite_energy_homeomorphisms_closed} follows from $C^0$-continuity of $\eta_d$.
\end{proof}

\begin{proof}[Proof of Theorem \ref{theorem:calabi_extension} (extension of Calabi)]
Let $\phi\in\on{Hameo}(\Sigma,\omega)$. By definition, we may choose a sequence of smooth Hamiltonians $H_j\in C^\infty(\BR/\BZ\times\Sigma)$ and a continuous Hamiltonian $H\in C^0(\BR/\BZ\times\Sigma)$ such that $H_j$ converges to $H$ with respect to the Hofer norm and $\phi_{H_j}^1$ converges to $\phi$ with respect to $d_{C^0}$. We define the extension of the Calabi invariant by
\begin{equation}
\label{eq:calabi_extension_proof_c}
\on{Cal}(\phi)\coloneqq \int_{\BR/\BZ\times \Sigma} H dt\wedge\omega.
\end{equation}
We need to verify that this is well-defined, i.e.\ does not depend on the choice of $H$ and $H_j$. We claim that
\begin{equation}
\int_{\BR/\BZ\times \Sigma} H dt\wedge\omega = \lim_{d\rightarrow\infty}\frac{\gamma_d(\phi)}{d}
\end{equation}
where we use the $C^0$-continuity from Theorem \ref{theorem:spectral_invariants_properties_diffeomorphisms_nonempty_boundary} to make sense of $\gamma_d(\phi)$. This clearly implies independence of choices. By the asymptotic property of $\gamma_d$ we have
\begin{equation}
\label{eq:calabi_extension_proof_a}
\lim_{d\rightarrow\infty}\frac{\gamma_d(\phi_{H_j}^1)}{d} = \on{Cal}(\phi_{H_j}^1)
\end{equation}
for every $j$. By $C^0$-continuity, we have
\begin{equation}
\label{eq:calabi_extension_proof_b}
\lim_{j\rightarrow\infty}\frac{\gamma_d(\phi_{H_j}^1)}{d} = \frac{\gamma_d(\phi)}{d}.
\end{equation}
In fact, this convergence is uniform in $d$ because $\frac{\gamma_d}{d}$ has Lipschitz constant $1$ with respect to the Hofer metric. Thus
\begin{equation}
\int_{\BR/\BZ\times \Sigma} H dt\wedge\omega =
\lim_{j\rightarrow\infty}\on{Cal}(\phi_{H_j}^1) =
\lim_{j\rightarrow\infty}\lim_{d\rightarrow\infty}\frac{\gamma_d(\phi_{H_j}^1)}{d} =
\lim_{d\rightarrow\infty}\lim_{j\rightarrow\infty}\frac{\gamma_d(\phi_{H_j}^1)}{d} =
\lim_{d\rightarrow\infty}\frac{\gamma_d(\phi)}{d}.
\end{equation}
Here the first equality uses that $H_j$ converges to $H$ in the Hofer norm, the second equality uses \eqref{eq:calabi_extension_proof_a}, the third equality uses uniformity of the convergence in \eqref{eq:calabi_extension_proof_b} and the last equality uses \eqref{eq:calabi_extension_proof_b}. This finishes our proof that the extension \eqref{eq:calabi_extension_proof_c} well-defined. It is straightforward to check that \eqref{eq:calabi_extension_proof_c} defines a homomorphism if it is well-defined (see \cite{Oh10} or \cite{CHMSS21}).
\end{proof}

\appendix

\section{Properties of the relative action spectrum}
\label{sec:properties_of_the_relative_action_spectrum}

In this section, we verify some basic properties of the relative action spectrum (Definition \ref{def:relative_action_spectrum}).

\begin{lem}
\label{lem:relative_action_spectrum_measure_zero}
The relative action spectrum is a measure zero subset of $\BR$.
\end{lem}

\begin{proof}
It is well known that action spectra of Hamiltonian diffeomorphisms have measure zero (see e.g. \cite{Oh02}). We adapt the proof of this fact given in \cite{Oh02} to our current setting. Let $\phi$ be an area preserving diffeomorphism of $(\Sigma,\omega)$. Let $d$ be a positive integer and let $\beta$ be a $d$-periodic orbit of $\phi$, not necessarily simple. We may regard $\beta$ as an orbit in the mapping torus $Y_\phi$. Let $N\subset Y_\phi$ be a small tubular neighbourhood of $\beta$. Let $\alpha$ be any other $d$-periodic orbit contained in $N$. Then there is a unique relative homology class $Z\in H_2(N,\alpha,\beta)$ and hence a well-defined {\it local action difference} $\langle Z,\omega_\phi\rangle$. The restriction of $\omega_\phi$ to $N$ is exact, i.e.\ we may pick a primitive $1$-form $\lambda$ on $N$. In terms of $\lambda$, the local action difference between $\alpha$ and $\beta$ is give by $\int_{\alpha}\lambda-\int_{\beta}\lambda$.\\
Let $q \in\Sigma$ be a point contained in the orbit $\beta$ and pick a small open neighbourhood $U \subset \Sigma$ of $q$. Given $p \in \Sigma$, let $\gamma_p:[0,d]\rightarrow Y_\phi$ be the trajectory segment of the flow on $Y_\phi$ starting at $[0,p]\in Y_\phi$. If we choose $U$ sufficiently small, then $\gamma_p$ is contained in $N$ for all $p\in U$. We fix a trivialization of $N$, i.e.\ an identification $N \cong S^1\times \BD$. For every $p \in U$, let $\delta_p$ denote the path in $N$ from $\gamma_p(d)$ to $p$ which is a straight line segment with respect to our trivialization of $N$. Then the concatenation $\gamma_p * \delta_p$ is a loop in $N$ for every $p\in U$. We define a smooth function $F:U\rightarrow\BR$ by
\begin{equation}
F(p)\coloneqq \int_{\gamma_p * \delta_p}\lambda - \int_\beta\lambda.
\end{equation}
It follows from a direct calculation that every point $p\in U$ which is a $d$-periodic point of $\phi$ is a critical point of $F$. Moreover, the value of $F$ at such a point is the local action difference between the $d$-periodic orbit $\alpha$ corresponding to $p$ and $\beta$. By Sard's theorem, the set of critical values of $F$ has measure zero. This shows that for every periodic orbit in $Y_\phi$ the set of local action differences with nearby orbits has measure zero. The assertion of the lemma is a consequence of this fact. Indeed, Let $H_\pm\in C^\infty(Y_\phi)$ be two Hamiltonians and fix $d\geq 1$. We may pick a finite collection of periodic orbits in $Y_{\phi_{H_\pm}}$ of period at most $d$ and tubular neighbourhoods of these orbits such that the set of local action differences in these tubular neighbourhoods has measure zero and every periodic orbit of $Y_{\phi_{H_\pm}}$ of period at most $d$ is contained in one of the tubular neighbourhoods. Clearly, the set of all values in $\on{Spec}_d(H_+,H_-)$ which only involve orbits in our chosen finite collection at most countable. Any other value in $\on{Spec}_d(H_+,H_-)$ can be obtained by adding finitely many local action differences. Since countable unions of measure zero sets have measure zero, this shows that the action spectrum $\on{Spec}_d(H_+,H_-)$ is a zero measure set.
\end{proof}

\begin{lem}
\label{lem:relative_action_spectrum_homotopy_invariance}
Consider Hamiltonians $H_\pm,H_\pm'\in C^\infty(Y_\phi)$. Suppose that the arcs $(\phi_{H_\pm}^t)_{t\in [0,1]}$ and $(\phi_{H_\pm'}^t)_{t\in [0,1]}$ in $\on{Ham}(\Sigma,\omega)$ are homotopic with fixed end points. Then for every positive integer $d\geq 1$ we have
\begin{equation}
\on{Spec}_d(H_+',H_-') = \on{Spec}_d(H_+,H_-) + dA^{-1}\int_{Y_\phi}((H_+'-H_+)-(H_-'-H_-))dt\wedge\omega_\phi
\end{equation}.
\end{lem}

\begin{proof}
By Lemma \ref{lem:shift_spectrum_energy}, it suffices to show that
\begin{equation}
\on{Spec}_d(H_+',H_-')=\on{Spec}_d(H_+,H_-)
\end{equation}
under the additional assumption that $H_\pm$ and $H_\pm'$ are mean normalized, i.e.\ satisfy \eqref{eq:normalization_hamiltonian}. In fact, it is not hard to further reduce to the case that $H_\pm$ and $H_\pm'$ satisfy the following stronger normalization condition:
\begin{equation}
\label{eq:normalization_hamiltonian_strong}
\int_{\pi^{-1}(t)} H \omega_\phi = 0 \qquad\text{for all}\enspace t\in\BR/\BZ
\end{equation}
where $\pi:Y_\phi\rightarrow\BR/\BZ$ denotes the natural projection. Lemma \ref{lem:relative_action_spectrum_homotopy_invariance} is a straightforward consequence of the following assertion:\\
{\it Let $(H^\lambda)_{\lambda\in [0,1]}\subset C^\infty(Y_\phi)$ be a smooth family of Hamiltonians satisfying the normalization \eqref{eq:normalization_hamiltonian_strong}. Assume that the time-$1$-map $\phi_{H^\lambda}^1$ is independent of $\lambda$. Let $\gamma:[0,1]\times S^1\rightarrow M_\phi$ be a smooth map such that for every $\lambda\in [0,1]$, the loop $\gamma^\lambda\coloneqq \gamma(\lambda,\cdot)$ parametrizes a periodic orbit on $\on{gr}(H^\lambda)$. Then
\begin{equation}
\label{eq:reltaive_action_spectrum_homotopy_invariance_proof_c}
\int_{[0,1]\times S^1} \gamma^*\Omega_\phi = 0.
\end{equation}}
We prove this statement by adapting an argument from \cite{Oh05}. Let us begin by showing that
\begin{equation}
\label{eq:relative_action_spectrum_homotopy_invariance_proof_a}
\frac{\partial}{\partial\lambda} \int_{[0,\lambda]\times S^1} \gamma^*\Omega_\phi = \int_{\gamma^\lambda} \partial_\lambda H^\lambda dt.
\end{equation}
We define the vector field $Y\coloneqq \partial_\lambda\gamma$ along the map $\gamma$. There is a unique splitting $Y=Y^v + Y^h$ such that $Y^v$ is tangent to the fibres $\BR\times \{p\}$ of $M_\phi$ and $Y^h(\lambda,\cdot)$ is tangent to $\on{gr}(H^\lambda)$. We compute
\begin{equation}
\frac{\partial}{\partial\lambda} \int_{[0,\lambda]\times S^1} \gamma^*\Omega_\phi = \int_{\gamma^\lambda} \iota_Y\Omega_\phi = \int_{\gamma^\lambda} \iota_{Y^v}\Omega_\phi + \iota_{Y^h}\Omega_\phi
\end{equation}
and
\begin{equation}
\iota_{Y^v}\Omega_\phi = \iota_{Y^v} (ds\wedge dt + \omega_\phi) = (\iota_{Y^v}ds) dt.
\end{equation}
Since the loop $\gamma^\lambda$ parametrizes a periodic orbit in $\on{gr}(H^\lambda)$, it is tangent to $\ker \Omega_\phi|_{\on{gr}(H^\lambda)}$. This implies that 
\begin{equation}
\int_{\gamma^\lambda}\iota_{Y^h}\Omega_\phi = 0
\end{equation}
and we obtain
\begin{equation}
\label{eq:relative_action_spectrum_homotopy_invariance_proof_d}
\frac{\partial}{\partial\lambda} \int_{[0,\lambda]\times S^1} \gamma^*\Omega_\phi = \int_{\gamma^\lambda}(\iota_{Y^v}ds) dt.
\end{equation}
Since the image of $\gamma^\lambda$ is contained in $\on{gr}(H^\lambda)$, we have
\begin{equation}
H^\lambda(p\circ\gamma^\lambda) = s\circ\gamma^\lambda
\end{equation}
where $p:M_\phi=\BR\times Y_\phi\rightarrow Y_\phi$ denotes the natural projection. Differentiating with respect to $\lambda$ yields
\begin{equation}
\label{eq:relative_action_spectrum_homotopy_invariance_proof_e}
\partial_\lambda H^\lambda(p\circ\gamma^\lambda) + dH^\lambda(p\circ\gamma^\lambda) \circ dp (Y) = ds(Y).
\end{equation}
Since $Y^h$ is tangent to $\on{gr}(H^\lambda)$, we have
\begin{equation}
\label{eq:relative_action_spectrum_homotopy_invariance_proof_f}
dH^\lambda(p\circ\gamma^\lambda)\circ dp (Y^h) = ds(Y^h).
\end{equation}
Subtracting \eqref{eq:relative_action_spectrum_homotopy_invariance_proof_f} from \eqref{eq:relative_action_spectrum_homotopy_invariance_proof_e} yields
\begin{equation}
\partial_\lambda H^\lambda(p\circ\gamma^\lambda) + dH^\lambda(p\circ\gamma^\lambda) \circ dp (Y^v)= ds(Y^v).
\end{equation}
We have $dp(Y^v)=0$ and hence
\begin{equation}
\partial_\lambda H^\lambda(p\circ\gamma^\lambda) = \iota_{Y^v}ds.
\end{equation}
Plugging this into \eqref{eq:relative_action_spectrum_homotopy_invariance_proof_d} yields \eqref{eq:relative_action_spectrum_homotopy_invariance_proof_a}.\\
Next, we prove that the right hand side of \eqref{eq:relative_action_spectrum_homotopy_invariance_proof_a} vanishes. This clearly implies \eqref{eq:reltaive_action_spectrum_homotopy_invariance_proof_c}. Following \cite{Oh05}, we define the family of Hamiltonian diffeomorphisms $\phi_t^\lambda\coloneqq \phi_{H^\lambda}^t\in\on{Ham}(\Sigma,\omega)$ parametrized by $t\in\BR$ and $\lambda\in [0,1]$. This gives rise to families of Hamiltonian vector fields on $(\Sigma,\omega)$ defined by
\begin{equation}
X_t^\lambda \coloneqq (\partial_t\phi_t^\lambda)\circ(\phi_t^\lambda)^{-1}\qquad\text{and}\qquad Y_t^\lambda \coloneqq (\partial_\lambda\phi_t^\lambda)\circ(\phi_t^\lambda)^{-1}.
\end{equation}
Let $F(\lambda,t,\cdot)$ and $K(\lambda,t,\cdot)$ be the unique mean normalized Hamiltonians on $(\Sigma,\omega)$ generating $X_t^\lambda$ and $Y_t^\lambda$, respectively. It is proved in \cite{Oh05} that $F$ and $K$ satisfy the key identity
\begin{equation}
\label{eq:relative_action_spectrum_homotopy_invariance_proof_b}
\partial_\lambda F = \partial_t K - \{F,K\}
\end{equation}
where $\{F,K\}=\omega(X_F,X_K)$ denotes the Poisson bracket. We use \eqref{eq:relative_action_spectrum_homotopy_invariance_proof_b} to evaluate the right hand side of \eqref{eq:relative_action_spectrum_homotopy_invariance_proof_a}. Let $d$ denote the period of $\gamma^\lambda$ and let $p\in\Sigma$ be a periodic point of $\phi_{H^\lambda}$ corresponding to the oribt $\gamma^\lambda$. Clearly, $F(\lambda,t,\cdot)=H^\lambda(\on{pr}(t,\cdot))$ where $\on{pr}:\BR\times\Sigma\rightarrow Y_\phi$ is the natural projection. Thus
\begin{IEEEeqnarray}{rCl}
\int_{\gamma^\lambda} \partial_\lambda H^\lambda dt  & = & \int_0^d (\partial_\lambda F)(\lambda,t,\phi_t^\lambda(p)) dt \nonumber\\
& = & \int_0^d \left((\partial_t K)(\lambda,t,\phi_t^\lambda(p)) + dK(\lambda,t,\phi_t^\lambda)X_t^\lambda\right) dt \nonumber\\
& = & \int_0^d \partial_t (K(\lambda,t,\phi_t^\lambda(p))) dt \nonumber\\
& = & K(\lambda,d,\phi_d^\lambda(p)) - K(\lambda,0,\phi_0^\lambda(p)) \nonumber\\
& = & 0.\nonumber
\end{IEEEeqnarray}
Here the second equality uses \eqref{eq:relative_action_spectrum_homotopy_invariance_proof_b}. The last equality uses the fact that $K(\lambda,m,\cdot)$ vanishes identically for every integer $m\in\BZ$ because $\phi_m^\lambda$ is independent of $\lambda$ by assumption. This concludes our proof of the lemma.
\end{proof}

\begin{lem}
\label{lem:relative_action_spectrum_closed}
Suppose that $\phi$ is rational (Definition \ref{def:rational_diffeomorphism}). Then the relative action spectrum $\on{Spec}_d(H_+,H_-)$ is a closed subset of $\BR$ for all $d\geq 1$ and all pairs of Hamiltonians $H_\pm$.
\end{lem}

\begin{proof}
Let $(\alpha_\pm^j)_j$ be sequences of orbit sets of $Y_{\phi_{H_\pm}}$ of degree $d$. Moreover, let $Z^j\in H_2(M_\phi,\alpha_+^j,\alpha_-^j)$ be a sequence of relative homology classes. Assume that the sequence $\langle Z^j,\Omega_\phi\rangle$ converges to a real number $A$. We need to show that $A\in\on{Spec}_d(H_+,H_-)$. After passing to a subsequence, we may assume that the sequences $\alpha_\pm^j$ converge to orbit sets $\alpha_\pm$. For $j$ sufficiently large, there are unique choices of relative homology classes $Z_\pm^j\in H_2(M_\phi,\alpha_\pm,\alpha_\pm^j)$ which are represented by chains contained in small tubular neighbourhoods of $\alpha_\pm$. We define $W^j\coloneqq Z_+^j+Z^j-Z_-^j\in H_2(M_\phi,\alpha_+,\alpha_-)$. Since $\alpha_\pm^j$ converge to $\alpha_\pm$, the sequence $\langle W^j,\Omega_\phi\rangle$ converges to $A$. We claim that in fact $\langle W^j,\Omega_\phi\rangle=A$ for $j$ sufficiently large. Indeed, rationality of $\phi$ implies that the image of $\langle \cdot,\Omega_\phi\rangle:H_2(M_\phi;\BZ)\rightarrow\BR$ is discrete. Thus the same is true for the image of $\langle \cdot,\Omega_\phi\rangle:H_2(M_\phi,\alpha_+\alpha_-)\rightarrow\BR$. A convergent sequence in a discrete set must be eventually constant, showing the claim and hence that $A$ is contained in the relative action spectrum.
\end{proof}

\section{Computation of Gromov-Taubes invariants}
\label{sec:computation_of_taubes_gromov_invariants}

\subsection{Review of Gromov-Taubes invariants}
\label{subsec:review_taubes_gromov_invariants}

Let $(X^4,\Omega)$ be a closed symplectic $4$-manifold. The Gromov-Taubes invariant introduced by Taubes in \cite{Tau96} is a function
\begin{equation}
\on{Gr}(X,\Omega;\cdot): H_2(X;\BZ)\rightarrow\BZ
\end{equation}
defined by counting pseudo-holomorphic submanifolds of $X$. Let $Z\in H_2(X;\BZ)$ be a homology class. The ECH index $I(Z)$ is defined by
\begin{equation}
I(Z)\coloneqq \langle Z,c_1(TX)\rangle + Z\cdot Z.
\end{equation}
It is always an even integer. If it is negative, $\on{Gr}(X,\Omega;Z)$ is defined to be zero. Assume that $I(Z)$ is non-negative and write $I(Z)=2k$ for a non-negative integer $k$. Let $P\subset X$ be a set of $k$ distinct points and let $J$ be a $\Omega$-compatible almost complex structure on $X$. The invariant $\on{Gr}(X,\Omega;Z)$ is defined to be a signed count of $J$-holomorphic submanifolds $C$ of $X$ passing through the points $P$ and representing the homology class $Z$. It is proved in \cite{Tau96} that this count is finite for generic $(J,P)$ and independent of choices. The following is a consequence of Gromov compactness.

\begin{lem}
\label{lem:gromov_taubes_nonvanishing_implies_curves}
Let $Z\in H_2(X;\BZ)$ and assume that $\on{Gr}(X,\Omega;Z)$ does not vanish. Then for every compatible almost complex structure $J$ on $X$ and every collection $P$ of $k = I(Z)/2$ distinct points in $X$, there exists a $J$-holomorphic curve passing through $P$ and representing $Z$.
\end{lem}

\begin{remark}
\label{remark:ignore_multiply_covered_tori}
Connected components of $C$ which are tori with trivial normal bundle are allowed to have positive integer multiplicities. Such components do not show up in our computations in subsection \ref{subsec:computations_without_seiberg_witten_theory}. We refer to \cite{Tau96} for details on how such multiply covered tori enter the count and do not further elaborate on this point.
\end{remark}

Suppose that $C\subset X$ is a $J$-holomorphic submanifold passing through $k$ points $P$. Moreover, assume that no component of $C$ is multiply covered (see Remark \ref{remark:ignore_multiply_covered_tori}). Let $N$ denote the normal bundle of $C$. Then there is a deformation operator
\begin{equation}
D_C : C^\infty(N)\rightarrow C^\infty (\Lambda^{0,1}T^*C\otimes N)
\end{equation}
which, roughly speaking, measures the failure of an infinitesimal deformation of $C$ to be $J$-holomorphic. Write $P=\{p_1,\dots,p_k\}$ and consider the composite operator
\begin{equation}
D_C\oplus \on{ev}_P : C^\infty(N)\rightarrow C^\infty (\Lambda^{0,1}T^*C\otimes N) \oplus \bigoplus\limits_{j=1}^k N_{p_j} \qquad s\mapsto (D_Cs,s(p_1),\dots s(p_k)).
\end{equation}
The operator $D_C\oplus \on{ev}_P$ is a Fredholm operator of index $0$ between suitable Sobolev completions. The submanifold $C$ is called {\it non-degenerate} if the operator $D_C\oplus\on{ev}_P$ is invertible (see \cite[Definition 2.1]{Tau96}). If $C$ is non-degenerate, it enters the count with sign $\pm 1$. This sign is determined as follows: Choose a path $(D^\lambda)_{\lambda\in [0,1]}$ of Cauchy-Riemann operators such that $D^0=D_C$, the operator $D^1$ is complex linear, and $D^1\oplus\on{ev}_P$ is invertible. Then the sign of $C$ is given by the mod-2 spectral flow of $(D^\lambda)_{\lambda\in [0,1]}$. In particular, the sign is $+1$ if $D_C$ happens to be complex linear, which is for example the case if $J$ is integrable.

\subsection{Computations without Seiberg-Witten theory}
\label{subsec:computations_without_seiberg_witten_theory}

In this section we perform elementary computations of Gromov-Taubes invariants which do not appeal to Seiberg-Witten theory.

\begin{theorem}
\label{theorem:taubes_gromov_projective_space}
Let $\Omega$ denote the Fubini-Study symplectic form on $\BC P^2$. Then
\begin{equation}
\on{Gr}(\BC P^2,\Omega;d\cdot [\BC P^1]) =
\begin{cases}
1 & \text{if}\enspace d\geq 0 \\
0 & \text{otherwise}.
\end{cases}
\end{equation}
\end{theorem}

\begin{theorem}
\label{theorem:taubes_gromov_sphere_cross_sphere}
Let $\Omega$ be a product symplectic form on $S^2\times S^2$. Then
\begin{equation}
\on{Gr}(S^2\times S^2,\Omega;c\cdot[S^2\times *] + d\cdot [*\times S^2]) =
\begin{cases}
1 & \text{if}\enspace c,d\geq 0 \\
0 & \text{otherwise}.
\end{cases}
\end{equation}
\end{theorem}

\begin{proof}[Proof of Theorem \ref{theorem:taubes_gromov_projective_space}]
We abbreviate $X\coloneqq \BC P^2$ and $Z\coloneqq [\BC P^1]\in H_2(X;\BZ)$. Let $d\in\BZ$. If $d<0$, it follows from the fact that any non-constant pseudo-holomorphic curve must have positive $\Omega$-area that $\on{Gr}(X,\Omega;dZ)=0$. If $d=0$, the empty set is the unique submanifold counted by $\on{Gr}$ and hence $\on{Gr}(X,\omega;0)=1$. Suppose that $d>0$. Then $I(dZ) = d(d+3) >0$. We set $k\coloneqq I(dZ)/2\in\BZ$. Let $i$ denote the standard integrable complex structure on $X$. Our goal is to show that there exists an open and dense subset $\MU\subset \on{Sym}_k(X)\setminus\Delta$ (here $\Delta$ denotes the fat diagonal) such that for all $P\in\MU$ the following is true:
\begin{itemize}
\item[($\ast$)] There exists exactly one $i$-holomorphic current in $X$ representing the homology class $dZ$ and passing through the points $P$. Moreover, this current consists of a single connected, embedded, non-degenerate $i$-curve of multiplicity $1$.
\end{itemize}
Before proving this assertion, let as argue that it indeed implies that $\on{Gr}(X,\Omega;dZ)=1$. The Gromov-Taubes invariant $\on{Gr}(X,\Omega;dZ)$ counts pseudo-holomorphic submanifolds for an admissible pair $(J,Q)$ consisting of a compatible almost complex structure $J\in \MJ(X,\Omega)$ and a set of $k$ points $Q\in\on{Sym}_k(X)\setminus\Delta$ (see \cite[Definition 4.2]{Tau96}). Let $P\in\MU$. In order to see that we may use $(i,P)$ to perform the count, we show that the count is constant as $(J,Q)$ ranges over a sufficiently small open neighbourhood $\MV\subset \MJ(X,\Omega)\times \MU$ of $(i,P)$. In fact, we claim that $\MV$ can be chosen such that ($\ast$) continues to hold for all $(J,Q)\in\MV$ in place of $(i,P)$. By the implicit function theorem, there exist a neighbourhood $\MV$ and a smooth family $C_{(J,Q)}$ of curves parametrized by $(J,Q)\in\MV$ and satisfying the properties in ($\ast$). We need to show that for $(J,Q)$ close to $(i,P)$, there are no other $J$-holomorphic currents representing $dZ$ and passing through $Q$. Assume by contradiction that this is false. Then we can find a sequence $(J_n,Q_n)$ converging to $(i,P)$ and a sequence of $J_n$-currents $C_n$ passing through $Q_n$ such that $[C_n]=dZ$ and $C_n\neq C_{(J_n,Q_n)}$. By Gromov compactness, we may pass to a subsequence such that $C_n$ converges to a $i$-current $C$ passing through $P$ and representing $dZ$. Since we assume that ($\ast$) holds for $(i,P)$, we must have $C=C_{(i,P)}$. Using that $C_{(i,P)}$ is embedded and non-degenerate, we deduce that $C_n=C_{(J_n,Q_n)}$ for $n$ sufficiently large, a contradiction. This shows that we may use $(i,P)$ to compute $\on{Gr}(X,\Omega;dZ)$. Since $i$ is integrable, the sign of the curve $C_{(i,P)}$ is $+1$. Hence $\on{Gr}(X,\Omega;dZ)=1$.\\
It remains to find an open and dense set $\MU$ such that ($\ast$) holds for all $P\in\MU$. We proceed in several steps.\\
{\bf Step 1:} {\it Let $\MO(d)$ denote the holomorphic line bundle on $X$ corresponding to the divisor $d\cdot\BC P^1$. We define $M\coloneqq \BP H^0(\MO(d))\cong \BC P^k$ to be the complex projectivization of the space of holomorphic global sections of $\MO(d)$. Then the set of $i$-holomorphic currents in $X$ representing $dZ$ is in bijective correspondence with the points of $M$.}\\
By the proper mapping theorem and Chow's theorem (see e.g. \cite{GH94}), $i$-holomorphic currents in $X$ are the same as effective divisors. If a $i$-holomorphic current in $X$ represents the homology class $dZ$, the holomorphic line bundle associated to the corresponding effective divisor is equal to $\MO(d)$. The assertion of step 1 now follows from the fact that effective divisors with associated line bundle $\MO(d)$ are in bijective correspondence with $\BP H^0(\MO(d))$. The space of sections $H^0(\MO(d))$ can be identified with the space of complex polynomials in three variables with are homogeneous of degree $d$. This space has complex dimension $\binom{d+2}{2}=k+1$. Hence $M\cong\BC P^k$.\\
{\bf Step 2:} {\it The set of all $f\in M$ such that the corresponding current in $X$ consists of a single connected and embedded curve with multiplicity $1$ is non-empty and Zariski open.}\\
It follows from Berini's theorem (see e.g. \cite{GH94}) that the $i$-holomorphic current corresponding to a generic $f\in M$ is embedded. Any embedded holomorphic current in $X$ must be connected for intersection reasons.\\
{\bf Step 3:} {\it We define the subvariety $V\subset M \times X^k$ by
\begin{equation}
V\coloneqq \{(f,p_1,\dots,p_k)\in M \times X^k \mid f \enspace \text{vanishes at all}\enspace p_j\}
\end{equation}
and abbreviate the natural projections by $\pi:V\rightarrow M$ and $\on{ev}:V \rightarrow X^k$. The complex dimension of $V$ is $2k$. The set of all $P\in X^k$ such that $\on{ev}^{-1}(P)$ consists of a single smooth point of $V$ which is also a regular point of $\on{ev}$ is non-empty and Zariski-open.}\\
We identify the space of global sections $H^0(\MO(d))$ with the $(k+1)$-dimensional space of all complex polynomials in three variables which are homogeneous of degree $d$. Requiring that a polynomial $f\in H^0(\MO(d))$ vanishes at some point $p\in X$ is equivalent to imposing a single homogeneous linear equation on the $k+1$ complex coefficients of $f$. Thus the condition that $f(p_j)=0$ for all $1\leq j\leq k$ is equivalent to $k$ homogeneous linear equations for the coefficients of $f$. If the tuple $(p_1,\dots,p_k)$ is generic (i.e.\ chosen from a non-empty Zariski open subset of $X^k$), then the solution space of this homogeneous linear system of $k$ equations for the $k+1$ coefficients of $f$ has complex dimension $1$. Therefore, there exists a unique $f\in\BP H^0(\MO(d))$ vanishing on $k$ generic points $p_1,\dots ,p_k$. This shows that $\on{ev}^{-1}(P)$ has cardinality $1$ for a generic $P\in X^k$. Since the singular locus of $V$ has complex dimension at most $2k-1$ and $X^k$ has complex dimension $2k$, the preimage $\on{ev}^{-1}(P)$ only consists of smooth points of $V$ for generic $P\in X^k$. A generic $P\in X^k$ must in addition be a regular value of $\on{ev}$. This finishes the proof of step 3.\\
{\bf Step 4:} Let us define $\MU\subset \on{Sym}_k(X)$ to be the image of the set of all points $P\in X^k$ such that
\begin{enumerate}
\item $P$ consists of $k$ distinct points in $X$.
\item $\on{ev}^{-1}(P)$ consists of a single smooth point of $V$ which is also a regular value of $\on{ev}$.
\item The current in $X$ corresponding to $\pi(\on{ev}^{-1}(P))$ consists of a single connected and embedded curve with multiplicity $1$.
\end{enumerate}
By the above steps, $\MU$ is non-empty and Zariski open in $\on{Sym}_k(X)$. We show that ($\ast$) holds for all $P\in\MU$. It is immediate from the above discussion that for every $P\in\MU$ there exists a unique $i$-holomorphic current in $X$ passing through $P$ and representing the homology class $dZ$. Moreover, this current consists of a single connected embedded holomorphic curve $C$ with multiplicity $1$. Let $N$ denote the normal bundle of $C$. The bundle $\Lambda^{0,1}T^*C\otimes N$ has first Chern number
\begin{equation}
\langle [C],c_1(\Lambda^{0,1}T^*C\otimes N) \rangle = \langle [C],c_1(TC\oplus N) \rangle = \langle [C],c_1(TX)\rangle = 3d
\end{equation}
which is positive. Thus we have automatic transversality (see \cite{HLS97}) and the operator $D_C$ must be surjective. It follows from the fact that $P$ is a regular value of $\on{ev}$, that the composite operator $D_C\oplus\on{ev}_P$ continues to be surjective. Since $D_C\oplus\on{ev}_P$ is an index-$0$ Fredholm operator (at least after passing to suitable Sobolev completions), it must be invertible, proving that $C$ is non-degenerate.
\end{proof}

\begin{proof}[Proof of Theorem \ref{theorem:taubes_gromov_sphere_cross_sphere}]
We follow the strategy of the proof of Theorem \ref{theorem:taubes_gromov_projective_space}. We equip $X=\BC P^1\times \BC P^1$ with the standard integrable complex structure $i$. Let $c,d\in\BZ$ be integers and abbreviate
\begin{equation}
Z\coloneqq c\cdot[S^2\times *] + d\cdot [*\times S^2] \in H_2(X;\BZ).
\end{equation}
The composition of any $i$-holomorphic curve in $X$ with the projection to either of the $\BC P^1$-factors is either constant or a holomorphic branched covering. This implies that if $c$ or $d$ are negative, then the homology class $Z$ is not represented by any $i$-holomorphic current in $X$. Hence the Gromov-Taubes invariant vanishes in this case. If $c=d=0$, the empty set is counted as the unique current representing $Z$ and hence $\on{Gr}(X,\Omega;0)=1$. Assume now that $c,d\geq 0$ and $c+d>0$. Then $I(Z) = 2c+2d+2cd >0$. We set $k\coloneqq I(Z)/2\in\BZ$. Our goal is to show that there exists an open and dense subset $\MU\subset\on{Sym}_k(X)\setminus\Delta$ such that for all $P\in\MU$ the following assertion is true:
\begin{itemize}
\item[($\ast$)] There exists exactly one $i$-holomorphic current in $X$ representing the homology class $Z$ and passing through the points $P$. Moreover, this current embedded and non-degenerate and all its connected components have multiplicity $1$.
\end{itemize}
The same argument given in the proof of Theorem \ref{theorem:taubes_gromov_projective_space} shows that ($\ast$) implies that $\on{Gr}(X,\Omega;Z)=1$. We show existence of an open and dense set $\MU$ satisfying ($\ast$) is several steps.\\
{\bf Step 1:} {\it For every integer $m\in\BZ$, let $\MO(m)$ denote the degree $m$ holomorphic line bundle on $\BC P^1$. For $j\in\{1,2\}$, let $\pi_j:X\rightarrow \BC P^1$ denote the projection onto the $j$-th factor. We define $M\coloneqq \BP H^0(\pi_1^*\MO(d)\otimes\pi_2^*\MO(c))\cong\BC P^k$ to be the projectivization of the space of holomorphic sections of the line bundle $\pi_1^*\MO(d)\otimes\pi_2^*\MO(c)$ on $X$. Then points in $M$ are in bijective correspondence with $i$-holomorphic currents in $X$ representing the homology class $Z$.}\\
Using the proper mapping theorem and Chow's theorem, we deduce that $i$-holomorphic currents in $X$ representing the homology class $Z$ are the same as effective divisors with associated holomoprhic line bundle $\pi_1^*\MO(d)\otimes\pi_2^*\MO(c)$. Such divisors are in bijection with the projectivization $M$. Holomorphic sections of $\pi_1^*\MO(d)\otimes\pi_2^*\MO(c)$ can be identified with complex polynomials in four variables $X_1,Y_1,X_2,Y_2$ which are homogeneous of degree $d$ in $X_1$ and $Y_1$ and homogeneous of degree $c$ in $X_2$ and $Y_2$. The space of such polynomials has complex dimension $(d+1)(c+1) = k+1$. Thus $M\cong\BC P^k$.\\
{\bf Step 2:} {\it The set of all $f\in M$ such that the corresponding current in $X$ is embedded and has only connected components of multiplicity $1$ is non-empty and Zariski open.}\\
This is immediate from Bertini's theorem (see e.g. \cite{GH94}).\\
{\bf Step 3:} {\it We define the subvariety $V\subset M \times X^k$ by
\begin{equation}
V\coloneqq \{(f,p_1,\dots,p_k)\in M \times X^k \mid f \enspace \text{vanishes at all}\enspace p_j\}
\end{equation}
and abbreviate the natural projections by $\pi:V\rightarrow M$ and $\on{ev}:V \rightarrow X^k$. The complex dimension of $V$ is $2k$. The set of all $P\in X^k$ such that $\on{ev}^{-1}(P)$ consists of a single smooth point of $V$ which is also a regular point of $\on{ev}$ is non-empty and Zariski-open.}\\
As mentioned in step 1, we can identify the space of global sections $H^0(\pi_1^*\MO(d)\otimes\pi_2^*\MO(c))$ with the $(k+1)$-dimensional space of all complex polynomials in four variables which are homogeneous of degree $d$ in the first two variables and homogeneous of degree $c$ in the last two. The remaining argument carries over from step 3 in the proof of Theorem \ref{theorem:taubes_gromov_projective_space} verbatim.\\
{\bf Step 4:} Let us define $\MU\subset \on{Sym}_k(X)$ to be the image of the set of all points $P\in X^k$ such that
\begin{enumerate}
\item $P$ consists of $k$ distinct points in $X$.
\item $\on{ev}^{-1}(P)$ consists of a single smooth point of $V$ which is also a regular value of $\on{ev}$.
\item The current in $X$ corresponding to $\pi(\on{ev}^{-1}(P))$ is embedded and all its connected components have multiplicity $1$.
\end{enumerate}
By the above steps, $\MU$ is non-empty and Zariski open in $\on{Sym}_k(X)$. We show that ($\ast$) holds for all $P\in\MU$. It is immediate from the above discussion that for every $P\in\MU$ there exists a unique $i$-holomorphic current in $X$ passing through $P$ and representing the homology class $Z$. Moreover, this current $C$ is embedded and all its connected components have multiplicity $1$. Let $N$ denote the normal bundle of $C$. The bundle $\Lambda^{0,1}T^*C\otimes N$ has first Chern number
\begin{equation}
\langle [C],c_1(\Lambda^{0,1}T^*C\otimes N) \rangle = \langle [C],c_1(TC\oplus N) \rangle = \langle [C],c_1(TX)\rangle = 2c+2d.
\end{equation}
Thus we have automatic transversality (see \cite{HLS97}) and the operator $D_C$ must be surjective. It follows from the fact that $P$ is a regular value of $\on{ev}$, that the composite operator $D_C\oplus\on{ev}_P$ continues to be surjective. Since $D_C\oplus\on{ev}_P$ is an index-$0$ Fredholm operator (at least after passing to suitable Sobolev completions), it must be invertible, proving that $C$ is non-degenerate.
\end{proof}

\subsection{Computations with Seiberg-Witten theory}
\label{subsec:computations_with_seiberg_witten_theory}

We give a brief review of Taubes' ``Seiberg-Witten = Gromov" theorem \cite{Tau00}. We recall that the Seiberg-Witten invariant of a closed, oriented $4$-manifold $X$ satisfying $b_2^+(X)\geq 2$ is a map
\begin{equation}
\on{SW}(X;\cdot) : \on{Spin^c}(X)\rightarrow\BZ
\end{equation}
assigning an integer $\on{SW}(X;\mathfrak{s})$ to every $\on{Spin^c}$ structure $\mathfrak{s}$ on $X$. The Seiberg-Witten invariant extends to the case $b_2^+(X)=1$. In this case the invariant in addition depends on the choice of one of two chambers, which is equivalent to the choice of an orientation of $H^{2+}(X;\BR)$. Fixing such an orientation results in two invariants $\on{SW}^\pm(X;\cdot)$. These two invariants are related by a wall crossing formula due to \cite{LL95} and \cite{OO96}.\\
A symplectic form $\Omega$ on $X$ determines a $\on{Spin^c}$ structure $\mathfrak{s}_\Omega$. Since $\on{Spin^c}(X)$ is a torsor over $H^2(X;\BZ)$, the distinguished $\on{Spin^c}$ structure $\mathfrak{s}_\Omega$ determines an identification
\begin{equation}
H^2(X;\BZ) \cong \on{Spin^c}(X) \quad E \leftrightarrow \mathfrak{s}_\Omega + E.
\end{equation}
In the case $b_2^+(X)\geq 2$, Taubes showed that
\begin{equation}
\on{Gr}(X,\Omega;Z) = \on{SW}(X;\mathfrak{s}_\Omega + \on{PD}(Z)).
\end{equation}
If $b_2^+(X)=1$, then the symplectic form $\Omega$ determines an orientation of $H^{2+}(X;\BR)$ and one has
\begin{equation}
\label{eq:taubes_gromov_equals_seiberg_witten_bplus_1}
\on{Gr}(X,\Omega;Z) = \on{SW}^+(X;\mathfrak{s}_\Omega + \on{PD}(Z)).
\end{equation}
The following computation of Taubes-Gromov invariants is based on \eqref{eq:taubes_gromov_equals_seiberg_witten_bplus_1}.

\begin{theorem}
\label{theorem:taubes_gromov_surface_cross_sphere}
Let $\Sigma$ be a closed, orientable surface of genus $g$ and endow the product $X\coloneqq \Sigma\times S^2$ with a product symplectic form $\Omega$. Then
\begin{equation}
\label{eq:taubes_gromov_surface_cross_sphere}
\on{Gr}(X,\Omega;c\cdot[\Sigma\times *] + d\cdot [*\times S^2]) =
\begin{cases}
(1+c)^g & \text{if}\enspace c\geq 0 \enspace\text{and}\enspace d\geq (g-1)\frac{c}{c+1} \\
0 & \text{otherwise}.
\end{cases}
\end{equation}
\end{theorem}

\begin{proof}
By \eqref{eq:taubes_gromov_equals_seiberg_witten_bplus_1} it suffices to compute $\on{SW}^+(X;\cdot)$. We refer to \cite{Sal99} for a detailed exposition of this computation and merely sketch some key elements. The main ingredient is a wall crossing formula due to \cite{LL95} and \cite{OO96}. It follows from this formula (see \cite[Corollary 1.4]{LL95}) that
\begin{equation}
\label{eq:wall_crossing_formula}
\on{SW}^+(X;\mathfrak{s}_\Omega+E) - \on{SW}^-(X;\mathfrak{s}_\Omega + E) = \left( 1 + \left\langle [*\times S^2],E \right\rangle\right)^g
\end{equation}
for all $E\in H^2(X;\BZ)$ satisfying $\langle [X],c_1(X)\cdot E + E\cdot E \rangle \geq 0$.\\
The fact that $\Sigma\times S^2$ admits a Riemannian metric of positive scalar curvature plays an essential role. It implies that for any $E\in H^2(X;\BZ)$ at least one of the invariants $\on{SW}^\pm(X;\mathfrak{s}_\Omega+E)$ vanishes. As explained in \cite{Sal99}, if the number $\langle [*\times S^2],E \rangle$ is non-negative, then $\on{SW}^-$ must vanish. If this number is less than $-1$ then $\on{SW}^+$ must vanish and if it is equal to $-1$ then both $\on{SW}^\pm$ vanish. This information together with \eqref{eq:wall_crossing_formula} and \eqref{eq:taubes_gromov_equals_seiberg_witten_bplus_1} readily yields \eqref{eq:taubes_gromov_surface_cross_sphere}.
\end{proof}

\bibliographystyle{plain}
\bibliography{elem_clos_lem}

\end{document}